\documentclass[12pt,draft]{article} 
\usepackage{amsmath,amssymb,amsthm,amsfonts, indentfirst}
\usepackage{enumerate,color,bm,graphicx,here}
\makeatletter
\def\@cite#1#2{[{{\bfseries #1}\if@tempswa , #2\fi}]}
\renewcommand{\section}{%
\@startsection{section}{1}{\z@}
{0.5truecm plus -1ex minus -.2ex}%
{1.0ex plus .2ex}{\bfseries\large}}
\def\@seccntformat#1{\csname the#1\endcsname.\ }
\makeatother

\numberwithin{equation}{section} 
\theoremstyle{theorem}
\newtheorem{thm}{Theorem}[section]

\newtheorem{lem}[thm]{Lemma}
\newtheorem{prop}[thm]{Proposition}
\theoremstyle{definition}
\newtheorem{df}{Definition}[section]
\newtheorem{remark}{Remark}[section]

\newcommand{\ep}{\varepsilon}
\newcommand{\pa}{\partial}

\newcommand{\ol}{\overline}

\begin{document}
\footnote[0]
    {2020{\it Mathematics Subject Classification}\/. 
    Primary: 35B44; %Blow-up in context of PDEs
    Secondary: 35K65, %Degenerate parabolic equations
                    92C17. %Cell movement (chemotaxis, etc.)
    }
     %chemotaxis; logistic source; finite-time blow-up; degenerate}
%==========================title==========================
\begin{center}
    \Large{{\bf Finite-time blow-up in a two species chemotaxis-competition model \\with
degenerate diffusion}}
\end{center}
\vspace{5pt}
%=========================author=========================
\begin{center}
    Yuya Tanaka\footnote{Partially supported by JSPS KAKENHI Grant Number JP22J11193.}
    %\footnote{Corresponding author}
    %Tomomi Yokota\footnote{Partially supported by Grant-in-Aid for
      %Scientific Research (C), No.\,21K03278.}
   \footnote[0]{
    E-mail: 
    {\tt yuya.tns.6308@gmail.com}
    %{\tt yokota@rs.tus.ac.jp}
    }\\
    \vspace{12pt}
    Department of Mathematics, 
    Tokyo University of Science\\
    1-3, Kagurazaka, Shinjuku-ku, 
    Tokyo 162-8601, Japan\\
    \vspace{2pt}
\end{center}
\begin{center}    
    \small \today
\end{center}

\vspace{2pt}
%=====================  Abstract  =======================
\newenvironment{summary}
{\vspace{.5\baselineskip}\begin{list}{}{%
     \setlength{\baselineskip}{0.85\baselineskip}
     \setlength{\topsep}{0pt}
     \setlength{\leftmargin}{12mm}
     \setlength{\rightmargin}{12mm}
     \setlength{\listparindent}{0mm}
     \setlength{\itemindent}{\listparindent}
     \setlength{\parsep}{0pt}
     \item\relax}}{\end{list}\vspace{.5\baselineskip}}
\begin{summary}
{\footnotesize {\bf Abstract.} This paper is concerned with the two-species chemotaxis-competition model with degenerate diffusion, 
  \[
  \begin{cases}
    u_t = \Delta u^{m_1} - \chi_1 \nabla\cdot(u\nabla w)
          + \mu_1 u (1-u-a_1v),
    &x\in\Omega,\ t>0,\\
    v_t = \Delta v^{m_2} - \chi_2 \nabla\cdot(v\nabla w)
            + \mu_2 v (1-a_2u-v),
    &x\in\Omega,\ t>0,\\
    0 = \Delta w +u+v-\ol{M}(t), 
    &x\in\Omega,\ t>0,
  \end{cases}
  \]
with $\int_\Omega w(x,t)\,dx=0$, $t>0$,
where $\Omega := B_R(0) \subset \mathbb{R}^n$ $(n\ge5)$ 
is a ball with some $R>0$;
$m_1,m_2>1$, $\chi_1,\chi_2,\mu_1,\mu_2,a_1,a_2>0$; 
$\ol{M}(t)$ is the spatial average of $u+v$.
The purpose of this paper is to show finite-time blow-up 
in the sense that there is $\widetilde{T}_{\rm max}\in(0,\infty)$ such that
  \[
    \limsup_{t \nearrow \widetilde{T}_{\rm max}}
    (\|u(t)\|_{L^\infty(\Omega)} + \|v(t)\|_{L^\infty(\Omega)})=\infty
  \]
for the above model within a concept of weak solutions fulfilling
a moment inequality which leads to blow-up.
To this end, we also give a result on finite-time blow-up 
in the above model with the terms $\Delta u^{m_1}$, $\Delta v^{m_2}$ 
replaced with the nondegenerate diffusion terms $\Delta (u+\delta)^{m_1}$, $\Delta (v+\delta)^{m_2}$, where $\delta\in(0,1]$.}
\end{summary}
\vspace{10pt}

\newpage

\section{Introduction}

\noindent{\bf Background.} A population dynamics of two species moving randomly and competing with each other is described as the diffusive Lotka--Volterra competition model  
  \[
    \begin{cases}
    u_t = \Delta u + \mu_1 u (1-u-a_1v),\\
    v_t = \Delta v + \mu_2 v (1-a_2u-v)
    \end{cases}
  \]
(see e.g.\ \cite{MurrayII}), where $\mu_1,\mu_2,a_1,a_2>0$ are constants and 
$u,v$ denote the population densities of the two species. 
In mathematical studies, it is known that solutions are always bounded. 
Moreover, results on stabilization were obtained 
(see e.g.\ \cite{Brown, CS_1977, MR_1979, KY_1993, KW_1985, Lou-Ni, Matano-Mimura}). 
Also, Pao \cite{Pao_2015} introduced 
the Lotka--Volterra competition model
with degenerate diffusion, 
  \[
    \begin{cases}
    u_t = \Delta u^{m_1} + \mu_1 u (1-u-a_1v),\\
    v_t = \Delta v^{m_2} + \mu_2 v (1-a_2u-v),
    \end{cases}
  \]
where $m_1,m_2\ge1$, and the author in \cite{Pao_2015} proved 
existence and uniqueness of classical solutions 
and asymptotic behavior. 
In \cite{HHWY_2020} a result 
on periodic solutions was obtained.
With regard to these models, 
the {\it linear} diffusion terms $\Delta u, \Delta v$ 
show the random walk of two species, 
while the {\it degenerate} diffusion terms
$\Delta u^{m_1}, \Delta v^{m_2}$ 
imply the movement of two species from 
high density region to low density region, 
which means that species avoid crowding.
Thus the model with degenerate diffusion 
provides a more realistic pattern of diffusion 
than that with linear diffusion for some species
(see e.g.\ \cite[Section 10.4]{Okubo-Leivin} and \cite[Section 11.3]{MurrayI}). 

Now, what will happen if the above model additionally has 
some terms playing opposite roles to degenerate diffusion, that is, 
promoting concentrations of population densities? 
Considering this question, we focus on the two-species 
chemotaxis-competition model with degenerate diffusion, 
  \begin{align}\label{model}
    \begin{cases}
      u_t = \Delta u^{m_1} - \chi_1 \nabla\cdot(u\nabla w)
            + \mu_1 u (1-u-a_1v),\\
      v_t = \Delta v^{m_2} - \chi_2 \nabla\cdot(v\nabla w)
              + \mu_2 v (1-a_2u-v),\\
      0 = \Delta w +u+v-w,
    \end{cases}
  \end{align}
where $\chi_1,\chi_2>0$ % and $\ol{M}(t)$ is the average of $u+v$ 
and $w$ represents a concentration of a signal substance. 
Here, this model with $m_1=m_2=1$ was proposed 
by Tello and  Winkler \cite{T-W_2012}, 
and the chemotaxis terms 
$- \chi_1 \nabla\cdot(u\nabla w), - \chi_2 \nabla\cdot(v\nabla w)$
imply that species move toward higher concentrations of the signal substance. 
Namely these chemtaxis terms work in the opposite way to the degenerate diffusion terms. 

Going back to the aforementioned question, 
we next give known results in \eqref{model} in the case of linear diffusion ($m_1=m_2=1$). 
In this case results on boundedness and stabilization were established 
under smallness conditions for $\chi_1$ and $\chi_2$ 
in \cite{BLM, M_2018_MMAS, STW, T-W_2012}; 
for details, in the cases that $a_1, a_2 \in (0,1)$ (\cite{BLM, T-W_2012}) 
and that $a_1>1>a_2$ (\cite{STW}) 
boundedness and stabilization toward spatially constant equilibria were 
shown under smallness conditions for $\chi_1$ and $\chi_2$; 
after that, in \cite{M_2018_MMAS} these conditions were improved 
in some cases; in particular, in \cite{M_2018_MMAS} 
the condition for boundedness was given by 
$\chi_1<\min\{ 1, a_1 \}\mu_1 \cdot \frac{n}{n-2}$
and 
$\chi_2<\min\{ 1,a_2 \}\mu_2 \cdot \frac{n}{n-2}$
(in \cite{MTY_2022} the same result was proved in a simplified model). 
In the model \eqref{model} with $0 = \Delta w +u+v-w$ 
replaced with $w_t = \Delta w +u+v-w$, 
similar results were obtained under smallness conditions for 
$\chi_1$ and $\chi_2$ in \cite{B-W, LMW_2015, M_2017_DCDSB}. 
On the contrary, in the case that $\chi_1$ and $\chi_2$ are 
large, it was established that blow-up occurs in the higher dimensional setting 
for a simplified model with linear diffusion
by the joint paper with Mizukami and Yokota \cite{MTY_2022}; 
indeed, in the case $m_1=m_2=1$, 
for the model \eqref{model} with $0 = \Delta w +u+v-w$ replaced with 
$0 = \Delta w +u+v-\ol{M}(t)$, where $\ol{M}(t)$ is the spatial average of $u+v$, 
it was shown that 
if $n\ge5$ and $\chi_1, \chi_2$ satisfy the largeness condition such that
  \[
    \chi_1>\frac{n}{n-4}\cdot\max\{ 1,a_1 \}\mu_1
    \quad \mbox{and} \quad \chi_2>\mu_2a_2, 
  \]
then there exists an initial data such that the corresponding solution 
blows up at some finite-time 
$T^\star$ in the sense that 
$\lim_{t \nearrow T^\star} 
( \| u(\cdot,t) \|_{L^\infty(\Omega)} + \| v(\cdot,t) \|_{L^\infty(\Omega)} ) = \infty$ 
in \cite[Section 5]{MTY_2022}. 
A related work on finite-time blow-up can be found in 
\cite[Section 4]{MTY_2022}, where a model \eqref{model} with 
$\mu_1 u (1-u-a_1v)$ and $\mu_2 v (1-a_2u-v)$ 
replaced with $\mu_1 u (1-u^{\kappa_1-1}-a_1v^{\lambda_1-1})$ 
and $\mu_2 v (1-a_2u^{\lambda_2-1}-v^{\kappa_2-1})$ ($\kappa_i,\lambda_i>1$, $i\in\{1,2\}$) was considered; 
in such a model, if $\kappa_2=\lambda_2=2$ and 
the third equation is $0 = \Delta w +u-w$, 
non-simultaneous blow-up occurs (see in \cite{MT_2023}).

From the above works, in the case of {\it linear} diffusion, 
results on boundedness and finite-time blow-up were established 
in some two-species chemotaxis-competition models. 
However, in the case of {\it degenerate} diffusion, 
to the best of our knowledge, there is no result in 
the corresponding models. 
Thus the following question arises: 
  \begin{center}
    Do solutions remain bounded or blow up in the case of {\it degenerate} diffusion? 
  \end{center}
The purpose of this paper is to give an answer on occurrence of blow-up to
a two-species chemotaxis-competition model. 

\medskip

\noindent{\bf Review of results on blow-up in chemotaxis systems.} Before considering the above question, we give an overview of known results on blow-up in the chemotaxis system with logistic source, 
  \begin{align}\label{modelm}
  \begin{cases}
    u_t = \Delta (u+\delta)^m - \chi \nabla\cdot(u\nabla w)
          + \lambda u - \mu u^\kappa,\\
    0 = \Delta w +u-\ol{M_u}(t)
\end{cases}
  \end{align}
where $\delta\in(0,1]$, $m\ge1$, $\chi,\lambda,\mu>0$, $\kappa>1$ 
and $\ol{M_u}$ is the spatial average of $u$. 
In the case $m=1$, Winkler \cite{W-2011} first found initial data such 
that the corresponding solution blows up in finite time 
under the smallness condition for $\kappa$ in the higher dimensional setting; a similar blow-up result was obtained in the three and four dimensional settings by Black, Fuest and Lankeit \cite{B-F-L}; 
furthermore, Fuest \cite{F_2021_optimal} improved the conditions for $\kappa$ in \cite{B-F-L, W-2011} to the condition that $1<\kappa<\min\left\{2,\frac{n}{2}\right\}$ and $\mu>0\ (n\ge3)$ and that $\kappa=2$ and $\mu\in\left(0,\frac{n-4}{n}\right)\ (n\ge5)$; 
in the two dimensional setting, finite-time blow-up occurs when $\int_\Omega u_0>8\pi$ in \eqref{modelm} with 
$\lambda u - \mu u^\kappa$ replaced with 
$\lambda(|x|) u - \mu(|x|) u^\kappa$ in \cite{F-2020}, 
and moreover, possible points of blow-up were investigated in 
\cite{BFLM_2022}.
In the case $m\ge1$, Black, Fuest and Lankeit \cite{B-F-L} 
showed blow-up of solutions to \eqref{modelm} with $\delta=1$ 
under the condition $m\in\big[1,2-\frac{2}{n}\big)$ and 
the smallness condition for $\kappa$; 
the other condition leading to blow-up was found in \cite{T_2022_JMAA}, 
which is the condition $2>\max\big\{m+\frac{2}{n}\kappa,\kappa\big\}$; 
moreover, in the case of degenerate diffusion ($\delta=0$) 
a solution blowing up was constructed 
under the same condition in \cite{T_2022_JMAA}.
In the system \eqref{modelm} with $0 = \Delta w +u-\ol{M_u}(t)$ 
replaced with $0 = \Delta w +u-v$, results on blow-up were 
found in \cite{B-F-L,T_2023_Eq,W-2018}; Winkler \cite{W-2018} 
proved that in the case $m=1$ blow-up occurs under the condition that  
$1<\kappa<\frac{7}{6}\ (n\in\{3,4\})$ and $1<\kappa<1+\frac{1}{2(n-1)}\ (n\ge5)$; 
in the case that $m\ge1$ and $\delta\in(0,1]$ 
a generalized condition leading to blow-up was derived in \cite{B-F-L}; 
also, in \cite{T_2023_Eq}, under the same condition in \cite{B-F-L} 
it was shown that there exists a solution blowing up in the case $\delta=0$.

\medskip
\noindent{\bf Main results.} Referring to \cite[Section 5]{MTY_2022}, 
we deal with the simplified two-species chemotaxis-competition model 
with $\delta\in(0,1]$ or $\delta=0$, 
\begin{align}\label{P0}
  \begin{cases}
    u_t = \Delta (u+\delta)^{m_1} - \chi_1 \nabla\cdot(u\nabla w)
          + \mu_1 u (1-u-a_1v),
    &x\in\Omega,\ t>0,\\
    v_t = \Delta (v+\delta)^{m_2} - \chi_2 \nabla\cdot(v\nabla w)
            + \mu_2 v (1-a_2u-v),
    &x\in\Omega,\ t>0,\\
    0 = \Delta w +u+v-\ol{M}(t), 
    &x\in\Omega,\ t>0,
%    \nabla u \cdot\nu = \nabla v \cdot\nu = \nabla w \cdot\nu = 0, 
%    &x\in\pa\Omega,\ t>0,\\
%    u(x,0)=u_0(x), \ \quad v(x,0)=v_0(x), 
%    &x\in\Omega, 
\end{cases}
\end{align}
with $\int_\Omega w(x,t)\,dx=0$, $t>0$,
where $\Omega = B_R(0) \subset \mathbb{R}^n$ $(n\ge5)$ 
is a ball with some $R>0$;
$m_1,m_2>1$, $\chi_1,\chi_2,\mu_1,\mu_2,a_1,a_2>0$;
  \[
    \ol{M}(t):=\frac{1}{|\Omega|}\int_\Omega (u(x,t)+v(x,t))\,dx
    \quad\mbox{for }t>0.
  \]

We recall the method to prove blow-up in \cite{MTY_2022}. 
In the case that $m_1=m_2=1$ and $\delta=0$, the key to the proof of blow-up 
was to derive the inequality 
$\phi_U'(t)\ge C\phi_U^2(t)$ with some $C>0$ 
in a framework of {\it classical} solutions, where 
$\phi_U$ is a moment-type functional. 
On the other hand, in the case of degenerate diffusion ($\delta=0$) 
the classical solutions are not always obtained, 
and so we are concerned with a framework of {\it weak} solutions in such a case. 
Thus the corresponding inequality for $\phi$ is not directly established 
in \eqref{P0} with $\delta=0$. 
To overcome this difficulty, we will derive an integral inequality for $\phi_U$ 
to show blow-up of solutions. 
To this end, we will also consider \eqref{P0} with nondegenerate diffusion 
$(\delta\in(0,1])$ as a problem approximate to the model with degenerate diffusion. 
Now we briefly state our main results as follows: 

\begin{itemize}
\item {\bf In the case of nondegenerate diffusion \boldmath{($\delta\in(0,1]$)} [Theorem \ref{thm1}]:} If 
  $n\ge5$ and 
    \begin{align}\label{introcondi1}
      m_1<2-\frac{4}{n}, \quad 
\chi_1>\frac{n(2-m_1)}{n(2-m_1)-4}\cdot\max\{ 1,a_1 \}\mu_1
      \quad \mbox{and} \quad \chi_2>\mu_2a_2
    \end{align}
  or
    \begin{align}\label{introcondi2}
      m_2<2-\frac{4}{n}, \quad 
      \chi_1>\mu_1a_1\quad \mbox{and} \quad 
      \chi_2>\frac{n(2-m_2)}{n(2-m_2)-4}\cdot\max\{ 1,a_2 \}\mu_2,
    \end{align}
then there exist initial data such that the corresponding classical solution
blows up at some finite-time $T^\star$ in the sense that 
$\lim_{t \nearrow T^\star} 
( \| u(\cdot,t) \|_{L^\infty(\Omega)} + \| v(\cdot,t) \|_{L^\infty(\Omega)} ) = \infty$.

\item {\bf In the case of degenerate diffusion \boldmath{($\delta=0$)} 
[Theorem \ref{thm2}]:} If $n\ge5$ 
and the condition \eqref{introcondi1} or \eqref{introcondi2} is satisfied, 
then there exist initial data such that the corresponding weak solution
blows up at some finite-time 
$\widetilde{T}_{\rm max}$ in the sense that 
$\limsup_{t \nearrow \widetilde{T}_{\rm max}}
(\|u(t)\|_{L^\infty(\Omega)} + \|v(t)\|_{L^\infty(\Omega)})=\infty$.
\end{itemize}

\medskip

\noindent{\bf Strategy of the proof.} In Section \ref{nondege} we consider the case of nondegenerate diffusion ($\delta\in(0,1]$). 
Following \cite{F_2021_optimal, W-2018}, we will introduce 
the moment-type functional 
$\phi_U(t) := \int^{s_0}_0 s^{-b}(s_0-s) U(s,t)\,ds$ for $t>0$ with
some $b\in(1,2)$ and $s_0\in(0,R^n)$, where
$U(s,t) := \int^{s^\frac{1}{n}}_0 \rho^{n-1} u(\rho,t)\,d\rho$ for 
$s\in[0,R^n]$ and $t>0$.
The goal is to derive the inequality 
$\phi_U'(t)\ge C\phi_U^2(t)$ with some $C>0$, which 
ensures that a maximal existence time is finite.
To obtain this inequality we will derive the key inequalities
  \[
    U_s(s,t)\le \frac{U(s,t)}{s}
    \quad\mbox{and}\quad
    V_s(s,t)\le \frac{V(s,t)}{s}
  \]
(see Lemma \ref{increase}), 
which enable us to control the terms 
coming from competition term. 
In the proof of Lemma \ref{increase} we will show from 
radially nonincreasing initial data that 
$U_s$ and $V_s$ are radially nonincreasing. 
The method is based on that of \cite[Lemma 5.2]{MTY_2022}, 
but it is the cornerstone of the proof to estimate integrals 
derived from the diffusion term.
%The key to the proof is to estimate integrals coming from the diffusion term. 
%by using the inequality such that for any $T>0$, 
%$u_r\le C(T) r$ for all $r\in[0,R]$ and $t\in(0,T)$.
%Next, through the mean value theorem the desired estimates are obtained. 

Section \ref{dege} is devoted to proving blow-up of solutions 
in the case of degenerate diffusion ($\delta=0$). 
In order to show finite-time blow-up we follow \cite{TY_2022_DCDS-B}, and so 
our goal is to construct a {\it moment solution}, 
which is a weak solution satisfying the moment inequality 
  \[
    \phi_U(t) - \phi_U(0) \ge K \int^t_0 \phi_U^2(\tau)\,d\tau
  \]
with some $K>0$. 
To this end, we consider an approximate problem with nondegenerate diffusion. 
Then, from Section \ref{nondege} we have 
the moment inequality for approximate solutions, 
$\phi_{U_\ep}(t) - \phi_{U_\ep}(0) \ge K \int^t_0 \phi_{U_\ep}^2(\tau)\,d\tau$ 
on $[0,T_{\rm max,\ep})$. Here, 
$\ep\in(0,1)$ is the approximate parameter and $T_{\rm max,\ep}$ is 
the maximal existence time for approximate solutions. 
As to this inequality, if we let $\ep\to0$ formally, then it seems to obtain 
the desired moment inequality. 
However, there is a possibility that $T_{\rm max,\ep}\to0$ as $\ep\to0$. 
The key to overcoming this difficulty is to find 
$T_0\in(0,\infty)$ and $K_0>0$ such that for any $\ep\in(0,1)$,
  \[
    T_0\le T_{\rm max,\ep}
    \quad\mbox{and}\quad
    \|u_\ep(\cdot,t)\|_{L^\infty(\Omega)} 
    + \|v_\ep(\cdot,t)\|_{L^\infty(\Omega)} \le K_0
    \quad\mbox{for all } t\in[0,T_0)
  \]
(see Lemma \ref{lem3.3}). Thanks to these inequalities, 
we can see that $T_{\rm max,\ep}$ does not vanish 
and be concerned with convergence of approximate solutions 
on $(0,T_0)$. 
Consequently, we can construct a moment solution.

\section{The case of nondegenerate diffusion}\label{nondege}
In this section we deal with \eqref{P0} with $\delta\in(0,1]$ as follows: 
\begin{align}\label{P1}
  \begin{cases}
    u_t = \Delta (u+\delta)^{m_1} - \chi_1 \nabla\cdot(u\nabla w)
          + \mu_1 u (1-u-a_1v),
    &x\in\Omega,\ t>0,\\
    v_t = \Delta (v+\delta)^{m_2} - \chi_2 \nabla\cdot(v\nabla w)
            + \mu_2 v (1-a_2u-v),
    &x\in\Omega,\ t>0,\\
    0 = \Delta w +u+v-\ol{M}(t), 
    &x\in\Omega,\ t>0,\\
    \nabla u \cdot\nu = \nabla v \cdot\nu = \nabla w \cdot\nu = 0, 
    &x\in\pa\Omega,\ t>0,\\
    u(x,0)=u_0(x), \ \quad v(x,0)=v_0(x), 
    &x\in\Omega, 
\end{cases}
\end{align}
with $\int_\Omega w(x,t)\,dx=0$, $t>0$,
where $\Omega = B_R(0) \subset \mathbb{R}^n$ $(n\ge5)$ 
is a ball with some $R>0$;
$m_1,m_2>1$, $\chi_1,\chi_2,\mu_1,\mu_2,a_1,a_2>0$; 
$u_0,v_0\in C^0(\ol{\Omega})$;
$\ol{M}(t):=\frac{1}{|\Omega|}\int_\Omega (u(x,t)+v(x,t))\,dx$ 
for $t>0$.

The main theorem in this section reads as follows.

%================================================================
%                                              Theorem 2.1                                              
%================================================================
%\newpage
\begin{thm}\label{thm1}
Let $\Omega = B_R(0) \subset \mathbb{R}^n$ $(n\ge5)$ 
and let $m_1,m_2>1$, $\chi_1, \chi_2, a_1, a_2, \mu_1,\mu_2>0$ 
and $\delta\in(0,1]$.
Assume that 
  \begin{align}\label{condim}
    m_1<2-\frac{4}{n} 
  \end{align}
and
  \begin{align}\label{condi}
    \chi_1>\frac{n(2-m_1)}{n(2-m_1)-4}\cdot\max\{ 1,a_1 \}\mu_1
    \quad \mbox{and} \quad \chi_2>\mu_2a_2.
  \end{align}
    Then, for all $M_0>0$ and $\widetilde{M}_0 \in (0,M_0)$ 
    there exists $r_0 \in (0,R)$ with the following property\/{\rm :}
    If $u_0$ and $v_0$ satisfy that 
  \begin{align}\label{initial1}
    u_0, v_0\in C^0(\ol{\Omega})\mbox{ are nonnegative, radially symmetric and radially nonincreasing}
  \end{align}
    as well as
      \begin{align}\label{initial2} 
        \int_\Omega u_0(x)\,dx \le M_0 
        \quad \mbox{and} \quad 
        \int_{B_{r_0}(0)} u_0(x)\,dx \ge \widetilde{M}_0, 
      \end{align}
    then the corresponding solution $(u,v,w)$ of \eqref{P1} 
    blows up at some finite time $T^\star>0$
    in the sense that 
      \begin{align}\label{blowupuv}
        \lim_{t \nearrow T^\star} 
        ( \| u(\cdot,t) \|_{L^\infty(\Omega)} + \| v(\cdot,t) \|_{L^\infty(\Omega)} )
        = \infty. 
      \end{align}
\end{thm}

\begin{remark}
If $m_2, \chi_1, \chi_2, a_1, a_2, \mu_1$ and $\mu_2$ satisfy
  \[
    m_2<2-\frac{4}{n} \quad\mbox{and}\quad 
    \chi_1>\mu_1a_1
    \quad\mbox{as well as}\quad
    \chi_2>\frac{n(2-m_2)}{n(2-m_2)-4}\cdot\max\{ 1,a_2 \}\mu_2,
  \]
then, by replacing $u_0$ with $v_0$ in \eqref{initial2} 
the same conclusion holds.
\end{remark}

\begin{remark}
In \eqref{blowupuv} it is not clear whether 
$u$ and $v$ blow up simultaneously or not.
\end{remark}

Before showing this theorem, we state a result on local existence of solutions to \eqref{P1}, 
which is obtained by a standard fixed point argument as in \cite{C-W, STW}.

%====================================================
%                                     Lemma 2.2                                     
%====================================================
\begin{lem}\label{LS}
Assume that $u_0,v_0\in C^0(\ol{\Omega})$ are nonnegative.
Then there exist $T_{\rm max} \in (0,\infty]$ and a unique triplet
$(u,v,w)$ of functions 
$u,v,w \in C^0(\overline{\Omega} \times [0,T_{\rm max})) 
                 \cap C^{2,1}(\overline{\Omega} \times (0,T_{\rm max}))$,
which solves \eqref{P1} with $\int_\Omega w(x,t)\,dx=0$, $t>0$ classically.
Also, $u,v\ge0$ in $\Omega \times (0,T_{\rm max})$ and 
  \[
    \mbox{if}\ T_{\rm max}<\infty, 
    \quad \mbox{then} \quad
    \lim_{t \nearrow T_{\rm max}} 
      ( \| u(\cdot,t) \|_{L^\infty(\Omega)} + \| v(\cdot,t) \|_{L^\infty(\Omega)} )
    = \infty.
  \]
Furthermore, if $u_0, v_0$ are radially symmetric, then so are $u,v,w$ 
for any $t \in (0,T_{\rm max})$.
\end{lem}
%====================================================

Through this section we fix $M_0>0$, 
$m_1,m_2>1$, $\chi_1,\chi_2,\mu_1,\mu_2,a_1,a_2>0$ and $\delta\in(0,1]$
as well as initial data $u_0$ and $v_0$ satisfying \eqref{initial1} and 
$\int_\Omega u_0\le M_0$.
Moreover, we denote the corresponding solution of \eqref{P1} 
provided by Lemma \ref{LS} by $(u,v,w)$ 
and its maximal existence time by $T_{\rm max}$.
Referring to \cite{F_2021_optimal, W-2018}, 
we define the mass accumulation functions $U$ and $V$ as 
  \begin{align}\label{UV}
    U(s,t) := \int^{s^\frac{1}{n}}_0 \rho^{n-1} u(\rho,t)\,d\rho,\quad 
    V(s,t) := \int^{s^\frac{1}{n}}_0 \rho^{n-1} v(\rho,t)\,d\rho
  \end{align}
for $s\in[0,R^n]$ and $t\in[0,T_{\rm max})$, 
and also introduce the moment-type functional $\phi_U$ as 	
  \begin{align}\label{phiU}
    \phi_U(t) := \int^{s_0}_{0} s^{-b}(s_0-s) U(s,t)\,ds
  \end{align}
and the functional $\psi_U$ as
  \[
    \psi_U(t) := \int^{s_0}_0 s^{-b}(s_0-s) U(s,t) U_s(s,t)\,ds
  \]
for $t \in [0,T_{\rm max})$ 
with suitably chosen $s_0 \in (0,R^n)$ and $b \in (1,2)$. 
The definition \eqref{UV} reveals that $U(0,t)=0$, so that 
by the mean value theorem we see that for any $t_0\in(0,T_{\rm max})$, 
$U(s,t)\le \|U_s\|_{L^\infty((0,R)\times(0,t_0))}\cdot s$ for all $s\in[0,R^n]$ and $t\in[0,t_0]$. 
Hence, as in \cite[Lemma 4.1]{F_2021_optimal}, we can make sure that 
$\phi_U\in C^0([0,T_{\rm max})) \cap C^1((0,T_{\rm max}))$.
We first state an inequality for $\phi_U'$ 
given by a straightforward computation.

%====================================================
%                                     Lemma 2.3                                     
%====================================================
\begin{lem}\label{lem2.3}
Let $s_0 \in (0,R^n)$ and $b \in (1,2)$. 
Then for all $t \in (0,T_{\rm max})$,
  \begin{align}\label{ineqphiU}
    \phi_U'(t) &\ge m_1n^2 \int^{s_0}_0 s^{2-\frac{2}{n}-b}(s_0-s) (nU_s(s,t)+\delta)^{m_1-1}U_{ss}(s,t)\,ds
                          + \chi_1 n
                             \psi_U(t)
  \notag \\ 
  &\quad\,
                          + \chi_1 n\int^{s_0}_0 s^{-b}(s_0-s) U_s(s,t)V(s,t)\,ds
                          - \chi_1 \overline{M}(t) \int^{s_0}_0 s^{1-b}(s_0-s) U_s(s,t)\,ds
  \notag \\ 
  &\quad\,
                          - \mu_1 n
                               \int^{s_0}_0 s^{-b}(s_0-s) 
                                 \left( \int^s_0 U_s^2(\sigma,t)\,d\sigma \right)\,ds
  \notag \\ 
  &\quad\,
                          - \mu_1 a_1 n 
                               \int^{s_0}_0 s^{-b}(s_0-s) 
                                 \left( \int^s_0 U_s(\sigma,t)V_s(\sigma,t)\,d\sigma \right)\,ds \notag \\
  &=: I_1+I_2+I_3+I_4+I_5+I_6.
  \end{align}
\end{lem}
%====================================================
\begin{proof}
Using the third equation in \eqref{P1}, we transform the first equation in \eqref{P1} exactly as in \cite[p.19]{MTY_2022}
(there only for linear diffusion) to arrive at \eqref{ineqphiU}.
%The first equation in \eqref{P1} is rewritten as 
%%
%  \begin{align}\label{ut} 
%    u_t = m_1r^{1-n}(r^{n-1} (u+\delta)^{m_1-1}u_r)_r 
%            - \chi_1 r^{1-n}(r^{n-1} u w_r)_r 
%            + \mu_1 u (1 - u - a_1 v). 
%  \end{align} 
%%
%for all $r\in(0,R)$ and $t\in(0,T_{\rm max})$. 
%Also, multiplying the third equation in \eqref{P1} by $r^{n-1}$ 
%and integrate it over $[0,s^{\frac{1}{n}}]$, 
%we have 
%%
%  \begin{align}\label{third}
%     s^{1-\frac{1}{n}}w_r = -U -V + \frac{1}{n}\ol{M}(t)s.
%  \end{align}
%%
%Since $U$ and $V$ defined in \eqref{UV} satisfy that
%that 
%$U_s(s,t) = \frac{1}{n} u(s^{\frac{1}{n}},t)$, 
%$V_s(s,t) = \frac{1}{n} v(s^{\frac{1}{n}},t)$
%and
%$U_{ss}(s,t) = \frac{1}{n^2} s^{\frac{1}{n}-1} u_r(s^{\frac{1}{n}},t)$
%for all $s=r^n\in(0,R^n)$ and $t\in(0,T_{\rm max})$, 
%we obtain from \eqref{ut} and \eqref{third} that
%%
%  \begin{align*}
%    U_t &= m_1n^2 s^{2-\frac{2}{n}} (nU_s+\delta)^{m_1-1}U_{ss} 
%              + \chi_1 n UU_s 
%              + \chi_1 n U_sV 
%              - \chi_1 n \ol{M}(t)sU_s
%  \\ 
%  &\quad\,
%              + \mu_1 U 
%              - \mu_1 n \int^s_0 U_s(\sigma,t)\,d\sigma 
%              - a_1 \mu_1 n \int^s_0 U_s(\sigma,t)V_s(\sigma,t)\,d\sigma.
%  \end{align*}
%%
%In light of \eqref{phiU} and the fact that $\mu_1 U\ge0$, 
%\eqref{ineqphiU} can be derived.
\end{proof}

In order to estimate $I_5$ and $I_6$ we derive the key inequalities 
$U_s\le \frac{U}{s}$ and $V_s\le \frac{V}{s}$. 
The method is based on \cite[Lemma 5.2]{MTY_2022}, but 
in our case one of the challenging terms to estimate arises from 
the diffusion term.

%====================================================
%                                     Lemma 2.4                                     
%====================================================
\begin{lem}\label{increase}
If $\chi_1$ and $\chi_2$ satisfy that 
  \begin{align}\label{mu1mu2}
    \chi_1>\mu_1a_1
    \quad \mbox{and} \quad
    \chi_2>\mu_2a_2,
  \end{align}
then 
  \begin{align}\label{mean}
    U_s(s,t) \le \frac{U(s,t)}{s}
      \quad \mbox{and} \quad 
    V_s(s,t) \le \frac{V(s,t)}{s}
  \end{align}
for all $s \in (0,R^n)$ and $t \in [0,T_{\rm max})$. 
\end{lem}
%====================================================
\begin{proof}
The goal is to show that $U_{ss}(s,t)\le0$ and $V_{ss}(s,t)\le0$ 
for all $s \in (0,R^n)$ and $t \in [0,T_{\rm max})$,
because \eqref{mean} is obtained   
by this inequalities and the mean value theorem.
Noting from \eqref{UV} that $U_{ss}=\frac{1}{n^2}s^{\frac{1}{n}-1}u_r$ and 
$V_{ss}=\frac{1}{n^2}s^{\frac{1}{n}-1}v_r$,
we prove that $u_r \le 0$ and $v_r \le 0$ 
for all $r \in (0,R)$ and $t \in (0,T_{\rm max})$. 
To this end, we establish the inequality 
  \[
    \frac{d}{dt}\int^R_0 r[(u_r)_+^2+(v_r)_+^2] \le \widetilde{c}\int^R_0r[(u_r)_+^2 + (v_r)_+^2],
  \]
where $\widetilde{c}>0$ depends on time.
From an approximation argument 
as in Step 2 in the proof of \cite[Lemma 2.2]{W-2018_nonlinear}, 
we can assume without loss of generality that 
$u_0, v_0 \in C^2(\overline{\Omega})$ with 
$\nabla u_0 \cdot \nu = \nabla v_0 \cdot \nu = 0$. 
Then we have 
  \[
    u_r, v_r \in C^0([0,R] \times [0,T_{\rm max}))\cap C^{2,1}((0,R) \times (0,T_{\rm max})).
  \]
As in \cite[p.\ 16]{MTY_2022}, from the first and third equations in \eqref{P1} 
it follows that
  \begin{align*}
    u_{rt} &= \left[r^{1-n}(r^{n-1} [(u+\delta)^{m_1}]_r)_r 
                 - \chi_1 r^{1-n}(r^{n-1} u w_r)_r
                 + \mu_1 u (1 - u - a_1 v)\right]_r
    \notag \\
           &= [(u+\delta)^{m_1}]_{rrr} + \frac{n-1}{r} [(u+\delta)^{m_1}]_{rr} - \frac{n-1}{r^2} [(u+\delta)^{m_1}]_{r}
    \notag \\
    &\quad\,
                - \chi_1 u_{rr} w_r - \chi_1 u_{r} w_{rr}
                + 2\chi_1 u u_r 
                + \chi_1 u_r v + \chi_1 u v_r 
                - \chi_1 \overline{M}(t) u_r
    \notag \\
    &\quad\,
                + \mu_1 u_r
                - 2 \mu_1 u u_r
                - a_1 \mu_1 u_r v
                - a_1 \mu_1 u v_r
  \end{align*}
for all $r\in(0,R)$ and $t\in(0,T_{\rm max})$. 
We rewrite this identity as 
  \begin{align*}
    u_t &= [(u+\delta)^{m_1}]_{rrr} + \frac{n-1}{r} [(u+\delta)^{m_1}]_{rr} 
             - \frac{n-1}{r^2} [(u+\delta)^{m_1}]_{r}\\
         &\quad\,
             + A_1(r,t)u_{rr} + B_1(r,t)u_r + D_1(r,t)v_r, 
  \end{align*}
where 
  \begin{align*}
    A_1(r,t)&:= - \chi_1 w_r, \\
    B_1(r,t)&:= - \chi_1 w_{rr} + 2\chi_1 u + \chi_1 v - \chi_1 \overline{M}(t) 
                 + \mu_1 - 2 \mu_1 u - a_1 \mu_1 v,\\
    D_1(r,t)&:= (\chi_1 - a_1 \mu_1) u
  \end{align*}
for $r\in(0,R)$ and $t\in(0,T_{\rm max})$.
Then we see that
  \begin{align}\label{ineq1}
    &\frac{1}{2}\frac{d}{dt}\int^R_0r(u_r)_+^2 \notag\\
    &\quad\,
    = \int^R_0 r(u_r)_+ [(u+\delta)^{m_1}]_{rrr}
       + (n-1) \int^R_0 (u_r)_+ [(u+\delta)^{m_1}]_{rr} \notag\\
    &\quad\,\quad\, 
       - (n-1) \int^R_0 \frac{1}{r}(u_r)_+ [(u+\delta)^{m_1}]_{r} \notag\\
    &\quad\,\quad\,
       + \int^R_0 r(u_r)_+ A_1(r,t)u_{rr}
       + \int^R_0 r(u_r)_+ B_1(r,t)u_r
       + \int^R_0 r(u_r)_+ D_1(r,t)v_r \notag\\
    &\quad\,
    =:J_1+J_2+J_3+J_4+J_5+J_6.
  \end{align}
Let $T\in(0,T_{\rm max})$. As in \cite[pp.\ 16--18]{MTY_2022}, under the condition \eqref{mu1mu2} 
the last three terms on the right-hand side are estimated as
  \begin{align}\label{ineq2}
    J_4+J_5+J_6
    \le c_1(T)\int^R_0 r[(u_r)_+^2 + (v_r)_+^2]
  \end{align}
for all $t\in(0,T)$ with some $c_1(T)>0$, so that we consider estimates for 
$J_1, J_2$ and $J_3$. 
First we can make sure that 
  \begin{align}\label{ineq3}
    J_3= - (n-1) m_1 \int^R_0 \frac{1}{r} (u+\delta)^{m_1-1} (u_r)_+^2\le0.
  \end{align}
Next, computing
  \begin{align*}
    [(u+\delta)^{m_1}]_{rr}
    =m_1(u+\delta)^{m_1-1}u_{rr} 
        + m_1(m_1-1)(u+\delta)^{m_1-2}u_r^2
  \end{align*}
and
  \begin{align*}
    [(u+\delta)^{m_1}]_{rrr}
    &=m_1(u+\delta)^{m_1-1}u_{rrr}
        + 3 m_1(m_1-1)(u+\delta)^{m_1-2}u_r u_{rr}\\
    &\quad\,
        + m_1(m_1-1)(m_1-2)(u+\delta)^{m_1-3}u_r^3, 
  \end{align*}
we can observe that 
  \begin{align}\label{ineq4}
    J_1+J_2
    &= m_1 \int^R_0 r(u_r)_+ (u+\delta)^{m_1-1}u_{rrr}
         + 3 m_1(m_1-1) \int^R_0 r (u+\delta)^{m_1-2}u_{rr} (u_r)_+^2 \notag\\
    &\quad\,
         + m_1(m_1-1)(m_1-2) \int^R_0 r (u+\delta)^{m_1-3}(u_r)^2(u_r)_+^2 \notag\\
    &\quad\,
         + (n-1)m_1 \int^R_0 (u_r)_+ (u+\delta)^{m_1-1}u_{rr} \notag\\
    &\quad\,
         + (n-1)m_1(m_1-1) \int^R_0 (u+\delta)^{m_1-2}u_r(u_r)_+^2.  
  \end{align}
Noting from regularity and radial symmetry that 
$u_r(0,t)=0$, we infer from the mean value theorem that 
$u_r(r,t)\le \|u_{rr}\|_{L^\infty((0,R)\times(0,T))}\cdot r$ 
for all $r\in(0,R)$ and $t\in(0,T)$. 
By using this inequality for the last term 
on the right-hand side of \eqref{ineq4}, 
we have
  \begin{align*}
    J_1+J_2 
    &\le m_1 \int^R_0 r(u_r)_+ (u+\delta)^{m_1-1}u_{rrr}
         + (n-1)m_1 \int^R_0 (u_r)_+ (u+\delta)^{m_1-1}u_{rr} \notag\\
    &\quad\,
         + c_2(T)\int^R_0 r (u_r)_+^2 \notag\\
    &=: J_7+J_8+ c_2(T)\int^R_0 r (u_r)_+^2
  \end{align*}
for all $t\in(0,T)$, where
  \begin{align*}
    c_2(T)
    &:= (n+2) m_1(m_1-1) \| u+\delta \|_{L^\infty((0,R)\times(0,T))}^{m_1-2}
        \| u_{rr} \|_{L^\infty((0,R)\times(0,T))} \\
    &\quad\,
       + m_1(m_1-1)(m_1-2)_+ \| u+\delta \|_{L^\infty((0,R)\times(0,T))}^{m_1-3}
          \| u_{r} \|_{L^\infty((0,R)\times(0,T))}^2.
  \end{align*}
Noticing from regularity and radial symmetry as well as 
the Neumann boundary condition in \eqref{P1} that 
$(u_r)_+(0,t) = (u_r)_+(R,t) = 0$, 
we see from integration by parts that 
  \begin{align}\label{ineq5}
    J_7+J_8
    &= - m_1 \int^R_0 (u_r)_+ (u+\delta)^{m_1-1}u_{rr} 
         - m_1 \int^R_0 r [(u_r)_+]_r (u+\delta)^{m_1-1}u_{rr} \notag\\
    &\quad\,
         - m_1(m_1-1) \int^R_0 r(u_r)_+ (u+\delta)^{m_1-2}u_r u_{rr} \notag\\
    &\quad\,
         + (n-1)m_1 \int^R_0 (u_r)_+ (u+\delta)^{m_1-1}u_{rr} \notag\\
    &= (n-2)m_1 \int^R_0 (u_r)_+ (u+\delta)^{m_1-1}u_{rr}
         - m_1 \int^R_0 r (u+\delta)^{m_1-1}[(u_r)_+]_r^2 \notag\\
    &\quad\,
         - m_1(m_1-1) \int^R_0 r (u+\delta)^{m_1-2}u_{rr}(u_r)_+^2 \notag\\
    &\le \frac{(n-2)}{2}m_1 \int^R_0 (u+\delta)^{m_1-1} [(u_r)_+^2]_r 
         - m_1(m_1-1) \int^R_0 r (u+\delta)^{m_1-2}u_{rr}(u_r)_+^2 \notag\\
    &=  -\frac{(n-2)}{2}m_1(m_1-1) \int^R_0 (u+\delta)^{m_1-2} (u_r)_+^3 \notag\\
    &\quad\,
          - m_1(m_1-1) \int^R_0 r (u+\delta)^{m_1-2}u_{rr}(u_r)_+^2 \notag\\
    &\le c_3(T)\int^R_0 r (u_r)_+^2
  \end{align}
for all $t\in(0,T)$, where 
$c_3(T):=m_1(m_1-1)\| u+\delta \|_{L^\infty((0,R)\times(0,T))}^{m_1-2}
\| u_{rr} \|_{L^\infty((0,R)\times(0,T))}$, so that 
a combination of \eqref{ineq1}--\eqref{ineq5} yields 
$\frac{d}{dt}\int^R_0 r(u_r)_+^2 \le c_4(T)\int^R_0r[(u_r)_+^2 + (v_r)_+^2]$ for all $t\in(0,T)$. 
Similarly, from the second and third equations in \eqref{P1} it follows that 
$\frac{d}{dt}\int^R_0 r(v_r)_+^2 \le c_5(T)\int^R_0r[(u_r)_+^2 + (v_r)_+^2]$ for all $t\in(0,T)$ 
with some $c_5(T)>0$. Thus we can obtain $c_6(T)>0$ such that
$\frac{d}{dt}\int^R_0r[(u_r)_+^2 + (v_r)_+^2] \le c_6(T)\int^R_0r[(u_r)_+^2 + (v_r)_+^2]$
for all $t\in(0,T)$, which together with the fact that 
$u_{0r}\le0$ and $v_{0r}\le0$ for all $r\in[0,R]$ ensures that 
$u_{r}\le0$ and $v_{r}\le0$ for all $r\in(0,R)$ and $t\in(0,T)$. 
Letting $T\nearrow T_{\rm max}$, we conclude the proof.
\end{proof}

Next we estimate the terms on the right hand side of \eqref{ineqphiU}.

%====================================================
%                                     Lemma 2.5                                     
%====================================================
\begin{lem}\label{lem2.5}
Let $s_0 \in (0,R^n)$. 
If $m_1$ satisfies \eqref{condim} and 
$\chi_1$ and $\chi_2$ fulfill \eqref{mu1mu2}, then 
there exist $b \in \big(1,2-\frac{4}{n(2-m_1)}\big)$ and $C>0$ such that
  \begin{align}\label{ineqphiU2}
    \phi_U'(t) &\ge - C s_0^{(3-b)\frac{2-m_1}{2}-\frac{2}{n}}\psi_U^{\frac{m_1}{2}}(t)
                          - Cs_0^{3-b-\frac{2}{n}}
                          + \chi_1 n \psi_U(t) 
  \notag\\
  &\quad\,
                          + \chi_1 n\int^{s_0}_0 s^{-b}(s_0-s) U_s(s,t) V(s,t)\,ds
                          - C s_0^{\frac{3-b}{2}} \sqrt{\psi_U(t)} 
                          - \frac{\mu_1 n}{b-1} \psi_U(t) 
  \notag \\ 
  &\quad\,
                          - \frac{a_1 \mu_1 n}{b-1}\int^{s_0}_0 s^{-b}(s_0-s) U_s(s,t) V(s,t)\,ds
  \end{align}
for all $t \in (0,T_{\rm max})$. 
\end{lem}
%====================================================
\begin{proof}
Since we can pick $b \in \big(1,2-\frac{4}{n(2-m_1)}\big)$ from \eqref{condim}, 
in the same way 
as in \cite[pp.\ 20, 21]{T_2022_JMAA} we find $c_1>0$ such that 
$I_1 \ge -c_1 s_0^{(3-b)\frac{2-m_1}{2}-\frac{2}{n}}\psi_U^{\frac{m_1}{2}}(t) - c_1s_0^{3-b-\frac{2}{n}}$
for all $t\in(0,T_{\rm max})$. 
By integrating the first equation in \eqref{P1} and using H\"{o}lder's inequality,  
we see that 
  \[
    \frac{d}{dt}\int_\Omega u \le \mu_1\int_\Omega u - \mu_1|\Omega|^{-1}\left(\int_\Omega u\right)^2,
  \]
which leads to $\int_\Omega u \le \max\big\{\int_\Omega u_0, |\Omega|\big\}$.
Since a similar estimate for $v$ is obtained, we find $c_2>0$ such that 
$\ol{M}(t) \le c_2$ for all $t\in(0,T_{\rm max})$. 
Noting this inequality and referring to an estimate for $\widetilde{I}_4$ in 
\cite[p.\ 20]{MTY_2022}, we have $I_4\ge- c_3 s_0^{\frac{3-b}{2}} \sqrt{\psi_U(t)}$ 
for all $t\in(0,T_{\rm max})$ with some $c_3>0$. 
Thanks to \eqref{mean}, we can employ \cite[Lemmas 5.4, 5.5]{MTY_2022}
to infer that $I_5$ and $I_6$ are estimated by 
the fifth and sixth terms on the right-hand side of \eqref{ineqphiU2}, respectively.
\end{proof}

\begin{remark}\label{rem}
In the proof of Lemma \ref{lem2.5} we applied an argument 
in \cite[pp.\ 20, 21]{T_2022_JMAA}. 
Thus $C>0$ provided by Lemma \ref{lem2.5} is independent of $\delta$.
\end{remark}

We establish  a super-quadratic nonlinear differential inequality for $\phi_U$.

%====================================================
%                                     Lemma 2.6                                     
%====================================================
\begin{lem}\label{lem2.6}
Let $s_0 \in (0,R^n)$. 
If \eqref{condim} and \eqref{condi} are satisfied, then 
there exist $b\in(1,2)$, $C_1>0$ and $C_2>0$ such that
  \begin{align}\label{ineqphiU3}
    \phi_U'(t) \ge C_1s_0^{-(3-b)} \phi_U^2(t) -C_2 s_0^{3-b-\frac{4}{n(2-m_1)}}
  \end{align}
for all $t\in(0,T_{\rm max})$. 
\end{lem}
%====================================================
\begin{proof}
By means of \eqref{condi} we can take 
$b\in\big(1+\frac{\max\{1,a_1\}\mu_1}{\chi_1},2-\frac{4}{n(2-m_1)}\big)$. 
Then the relation $b>1+\frac{\max\{1,a_1\}\mu_1}{\chi_1}$ guarantees that 
$\chi_1 - \frac{\max\{1,a_1\}\mu_1}{b-1}>0$.
By virtue of this and Young's inequality, it follows from \eqref{ineqphiU2} that for any $s_0\in(0,R^n)$,
  \begin{align*}
    \phi_U'(t) 
    &\ge - c_1 s_0^{(3-b)\frac{2-m_1}{2}-\frac{2}{n}}\psi_U^{\frac{m_1}{2}}(t) 
            - c_1 s_0^{3-b-\frac{2}{n}}
            + \left(\chi_1 - \frac{\mu_1}{b-1}\right) n \psi_U(t) \notag\\
    &\quad\,
            + \left(\chi_1 - \frac{\mu_1a_1}{b-1}\right) n\int^{s_0}_0 s^{-b}(s_0-s) U_s(s,t) V(s,t)\,ds
            - c_1 s_0^{\frac{3-b}{2}} \sqrt{\psi_U(t)} \notag\\
    &\ge - c_1 s_0^{(3-b)\frac{2-m_1}{2}-\frac{2}{n}}\psi_U^{\frac{m_1}{2}}(t)
            - c_1 s_0^{3-b-\frac{2}{n}}
            + \left(\chi_1 - \frac{\mu_1}{b-1}\right) n \psi_U(t)
            - c_1 s_0^{\frac{3-b}{2}} \sqrt{\psi_U(t)} \notag\\
    &\ge \frac{1}{2}\left(\chi_1 - \frac{\mu_1}{b-1}\right) \psi_U(t) 
            - c_2s_0^{3-b-\frac{4}{n(2-m_1)}} 
            - c_1 s_0^{3-b-\frac{2}{n}} - c_2 s_0^{3-b} 
  \end{align*}
with some $c_1,c_2>0$. 
As to the last three terms on the right-hand side of this inequality, 
the relation $s_0<R^n$ implies that 
  \[ 
    c_2s_0^{3-b-\frac{4}{n(2-m_1)}} + c_1 s_0^{3-b-\frac{2}{n}} + c_2 s_0^{3-b}
    < c_3s_0^{3-b-\frac{4}{n(2-m_1)}}, 
  \]
where $c_3:=\big(c_2+c_1(R^n)^{\frac{2}{n}\cdot\frac{m_1}{2-m_1}}+c_2(R^n)^\frac{4}{n(2-m_1)}\big)$,
and thus, 
  \[
    \phi_U'(t) \ge \frac{1}{2}\left(\chi_1 - \frac{\mu_1}{b-1}\right) 
                        \psi_U(t) -c_3 s_0^{3-b-\frac{4}{n(2-m_1)}}
  \]
for all $t\in(0,T_{\rm max})$. 
Finally, we rely on \cite[Lemma 4.3]{W-2018} to obtain \eqref{ineqphiU3}.
\end{proof} 

We next derive the key inequality for $\phi_U$, which will ensures that $T_{\rm max}<\infty$.

%====================================================
%                                     Lemma 2.7                                     
%====================================================
\begin{lem}\label{lem2.7} 
If \eqref{condim} and \eqref{condi} are satisfied,
then for all $M_0>0$ and $\widetilde{M}_0 \in (0,M_0)$ 
one can find $r_0=\big(\frac{s_0}{4}\big)^\frac{1}{n} \in (0,R)$ with the following property\/{\rm :}
If $u_0$ and $v_0$ fulfill \eqref{initial1} and \eqref{initial2}, 
then there exists $C>0$ such that
  \begin{align}\label{ineqphiU4}
    \phi_U'(t) \ge Cs_0^{-(3-b)} \phi_U^2(t)
  \end{align}
for all $t\in(0,T_{\rm max})$. 
\end{lem}
%====================================================
\begin{proof}
By virtue of Lemma \ref{lem2.6} we can take 
$b\in(1,2)$, $c_1>0$ and $c_2>0$ such that for any $s_0\in(0,R^n)$, 
$\phi_U'(t) \ge c_1s_0^{-(3-b)} \phi_U^2(t) -c_2 s_0^{3-b-\frac{4}{n(2-m_1)}}$ 
for all $t\in(0,T_{\rm max})$. 
Now let us fix $s_0\in(0,R^n)$ fulfilling that 
  \begin{align}\label{s0}
    s_0 \le \left(\frac{c_1}{2} \cdot \left(\frac{2^{2-b}\widetilde{M}_0}{\omega_n}\right)^2 \cdot \frac{1}{c_2}\right)^{\frac{1}{2-\frac{4}{n(2-m_1)}}}, 
  \end{align}
where $\omega_n$ denotes the $(n-1)$-dimensional surface area 
of the unit sphere in $\mathbb{R}^n$. 
We introduce the set $S$ as 
  \[
    S:=\left\{T\in(0,T_{\rm max}) \;\middle|\; \phi_U(t)\ge\frac{2^{b-3}\widetilde{M}_0}{\omega_n}s_0^{2-b}\quad\mbox{for all}\ t\in[0,T]\right\}.
  \]
Then $S$ is not empty. Indeed, 
putting $r_0:=\big(\frac{s_0}{4}\big)^\frac{1}{n}$, 
we see from \cite[(5.5)]{W-2018} that 
  \[
    \phi_U(0)>\frac{2^{b-3}\widetilde{M}_0}{\omega_n}s_0^{2-b}, 
  \]
which together with continuity of $\phi_U$ guarantees that $S \neq \emptyset$. 
Therefore we can define $T_1:=\sup S\in(0,T_{\rm max}]$. 
The definition of $S$ and \eqref{s0} entail that 
  \begin{align*}
    \frac{c_1}{2}s_0^{-(3-b)} \phi_U^2(t) -c_2 s_0^{3-b-\frac{4}{n(2-m_1)}}
    &\ge \frac{c_1}{2} \left(\frac{2^{b-3}\widetilde{M}_0}{\omega_n} \right)^2 s_0^{1-b}
         - c_2 s_0^{3-b-\frac{4}{n(2-m_1)}} \\
    &= \left[\frac{c_1}{2} \left(\frac{2^{b-3}\widetilde{M}_0}{\omega_n} \right)^2 
         - c_2 s_0^{2-\frac{4}{n(2-m_1)}}\right] s_0^{1-b}
    \ge 0, 
  \end{align*}
and thereby it follows that $\phi_U'(t) \ge \frac{c_1}{2}s_0^{-(3-b)} \phi_U^2(t)\ge0$
for all $t\in(0,T_1)$. This leads to $T_1=T_{\rm max}$, 
which concludes the proof.
\end{proof}

Finally we prove Theorem \ref{thm1}.

%====================================================
%                            Proof of Theorem 2.1                               
%====================================================
\begin{proof}[Proof of Theorem \ref{thm1}]
By means of Lemma \ref{lem2.7}, the inequality \eqref{ineqphiU4} holds, 
which yields $T_{\rm max}<\infty$. 
Therefore the claim is a direct consequence of Lemma \ref{LS}.
\end{proof}

%================================================================
%                                                                                                             
%                                                Section 3                                                
%                                                                                                             
%================================================================
\section{The case of degenerate diffusion}\label{dege}

In this section we consider the system \eqref{P0} with $\delta=0$ as follows:
\begin{align}\label{P2}
  \begin{cases}
    u_t = \Delta u^{m_1} - \chi_1 \nabla\cdot(u\nabla w)
          + \mu_1 u (1-u-a_1v),
    &x\in\Omega,\ t>0,\\
    v_t = \Delta v^{m_2} - \chi_2 \nabla\cdot(v\nabla w)
            + \mu_2 v (1-a_2u-v),
    &x\in\Omega,\ t>0,\\
    0 = \Delta w +u+v-\ol{M}(t), 
    &x\in\Omega,\ t>0,\\
    \nabla u^{m_1} \cdot\nu = \nabla v^{m_2} \cdot\nu = \nabla w \cdot\nu = 0, 
    &x\in\pa\Omega,\ t>0,\\
    u(x,0)=u_0(x), \ \quad v(x,0)=v_0(x), 
    &x\in\Omega, 
\end{cases}
\end{align}
with $\int_\Omega w(x,t)\,dx=0$, $t>0$,
where $\Omega = B_R(0) \subset \mathbb{R}^n$ $(n\ge5)$;
$m_1,m_2>1$ and $\chi_1,\chi_2,\mu_1,\mu_2,a_1,a_2>0$; 
$u_0,v_0\in L^\infty(\ol{\Omega})$;
$\ol{M}(t)=\frac{1}{|\Omega|}\int_\Omega (u(x,t)+v(x,t))\,dx$
for $t>0$.

In the case of nondegenerate diffusion, 
the differential inequality \eqref{ineqphiU4} is the key to the proof.
However, we cannot derive it in the case of degenerate diffusion 
due to the lack of regularity of solutions. 
To overcome this difficulty, we aim to construct a weak solution satisfying 
an integral inequality, which is defined in \eqref{mineq}.
Before stating a main result, we introduce definitions of 
{\it moment solutions}, 
{\it maximal moment solutions} 
and {\it blow-up} for moment solutions.

%====================================================
%                       Definition of moment solutions                       
%====================================================
\begin{df}[{moment solutions}]\label{ms}
Let $T \in (0,\infty]$. 
A triplet $(u,v,w)$ of nonnegative and radially symmetric functions defined on $\Omega\times(0,T)$
is called a {\it moment solution} of \eqref{P2} 
on $[0,T)$ if
\begin{enumerate}[(i)]
\item $u,v \in L^\infty_{\rm loc}([0,T);L^\infty(\Omega))\cap C^0_{{\rm w}-\star}([0,T);L^\infty(\Omega))$
and,

$u^{m_1}, v^{m_2} \in L^2(0,T;H^1(\Omega))$ \ if $T<\infty$; \
$u^{m_1}, v^{m_2} \in L^2_{\rm loc}([0,T);H^1(\Omega))$ \ if $T=\infty$,
\item $w \in L^\infty_{\rm loc}([0,T);H^{1}(\Omega))$,
\item for all $\varphi \in L^2(0,T;H^1(\Omega)) \cap W^{1,1}(0,T;L^2(\Omega))$
with ${\rm supp}\, \varphi(x,\cdot)\subset[0,T)$ for a.a.\ $x\in\Omega$,
\begin{align*}
  &\int^{T}_{0} \int_\Omega 
    (\nabla u^{m_1} \cdot \nabla \varphi - \chi_1 u \nabla w \cdot \nabla \varphi
     - \mu_1 u (1-u-a_1v) \varphi 
     - u \varphi_t)
  = \int_\Omega u_0 \varphi(\cdot,0),
  \\
  &\int^{T}_{0} \int_\Omega 
    (\nabla v^{m_2} \cdot \nabla \varphi - \chi_2 v \nabla w \cdot \nabla \varphi
     - \mu_2 v (1-a_2u-v) \varphi 
     - v \varphi_t)
  = \int_\Omega v_0 \varphi(\cdot,0),
  \\
  &\int^{T}_{0} \int_\Omega \nabla w \cdot \nabla \varphi
  - \int^{T}_{0} \int_\Omega (u+v) \varphi 
  + \int^{T}_{0} \left(\ol{M}(t) \int_\Omega \varphi \right)
  =0,
\end{align*}
\item $(u,v)$ satisfies the following moment inequality: 
\begin{align}\label{mineq}
  \phi_U(t) - \phi_U(0) \ge K \int^t_0 \phi_U^2(\tau)\,d\tau
  \quad \mbox{for all}\ t\in(0,T) 
\end{align}
for some constant $K>0$ independent of $T$, where $\phi_U$ is the moment-type
functional defined in \eqref{phiU}.
\end{enumerate}
\end{df}

%====================================================
%                Definition of maximal moment solutions                  
%====================================================
\begin{df}[{maximal moment solutions}]\label{mms}
Define the set $\mathcal{S}$ as
  \[
    \mathcal{S}:=\{ (T,u,v,w) \;|\; 
                          T\in(0,\infty],\ 
                          (u,v,w) \ \mbox{is a moment solution of \eqref{P2} on}\ [0,T)
                       \},
  \]
which is not empty by Proposition \ref{LMS}, 
with the order relation $\preceq$ given by
  \begin{align*}
    &(T_1,u_1,v_1,w_1) \preceq (T_2,u_2,v_2,w_2) \\ 
    &\quad\,:\Longleftrightarrow \ 
    T_1 \le T_2,\ u_2|_{(0,T_1)}=u_1,\ v_2|_{(0,T_1)}=v_1\ w_2|_{(0,T_1)}=w_1.
  \end{align*}
Then Zorn's lemma assures some maximal element 
$(\widetilde{T}_{\rm max},u,v,w)\in\mathcal{S}$, 
and $(u,v,w)$ is called a {\it maximal moment solution} of \eqref{P2} on $[0,\widetilde{T}_{\rm max})$.
\end{df}

%====================================================
%                             Definition of blow-up                               
%====================================================
\begin{df}[{blow-up}]
Let $(u,v,w)$ be a maximal moment solution of \eqref{P2} 
on $[0,\widetilde{T}_{\rm max})$.
If $u$ and $v$ satisfy
  \[
    \underset{t \nearrow \widetilde{T}_{\rm max}}{\limsup}\, (\|u(t)\|_{L^\infty(\Omega)} + \|v(t)\|_{L^\infty(\Omega)})
    =\infty,
  \]
then we say that $(u,v,w)$ {\it blows up} at $\widetilde{T}_{\rm max}$.
\end{df}

The main theorem in this section reads as follows.

%================================================================
%                                              Theorem 2.2                                              
%================================================================
\begin{thm}\label{thm2}
Let $\Omega = B_R(0) \subset \mathbb{R}^n$ $(n\ge5)$, 
$m_1,m_2>1$ and $\chi_1, \chi_2, a_1, a_2, \mu_1,\mu_2>0$.
Assume that \eqref{condim} and \eqref{condi} hold.
Then, for all $M_0>0$ and $\widetilde{M}_0 \in (0,M_0)$ 
there exists $r_0 \in (0,R)$ with the following property\/{\rm :}
If $u_0$ and $v_0$ satisfy that 
  \begin{align}\label{initial3}
    u_0, v_0\in L^\infty(\Omega)\mbox{ are nonnegative, radially symmetric and radially nonincreasing}
  \end{align}
and that \eqref{initial2} holds,
then there exists $\widetilde{T}_{\rm max}<\infty$ and a moment solution $(u,v,w)$ of \eqref{P2} 
blows up at $\widetilde{T}_{\rm max}$.
\end{thm}
%================================================================

%====================================================
%                                                                                                             
%                                  Subsection 3.1                                  
%                                                                                                             
%====================================================
\subsection{Local existence of moment solutions}

The purpose of this section is to prove local existence of moment solutions.

%====================================================
%                                  Proposition 3.2                                  
%====================================================
\begin{prop}\label{LMS}
Let $\Omega = B_R(0) \subset \mathbb{R}^n$ $(n\ge5)$ 
and let $m_1,m_2>1$, $\chi_1, \chi_2, a_1, a_2, \mu_1,\mu_2>0$.
Assume that \eqref{condim} and \eqref{condi} hold.
Then, for all $M_0>0$ and $\widetilde{M}_0 \in (0,M_0)$ 
there exists $r_0 \in (0,R)$ with the following property\/{\rm :}
If $u_0$ and $v_0$ satisfy \eqref{initial3} and \eqref{initial2}, then there exist 
$\widetilde{T}_{\rm max}\in(0,\infty]$
and a moment solution of \eqref{P2} on $[0,\widetilde{T}_{\rm max})$.
\end{prop}

To show this proposition
we deal with a problem approximate to the system \eqref{P1}, 
\begin{align}\label{P3}
  \begin{cases}
    (u_\ep)_t = \Delta (u_\ep+\ep)^{m_1} - \chi_1 \nabla\cdot(u_\ep \nabla w_\ep)
          + \mu_1 u_\ep (1-u_\ep-a_1v_\ep),
    &x\in\Omega,\ t>0,\\
    (v_\ep)_t = \Delta (v_\ep+\ep)^{m_2} - \chi_2 \nabla\cdot(v_\ep\nabla w_\ep)
            + \mu_2 v_\ep (1-a_2u_\ep-v_\ep),
    &x\in\Omega,\ t>0,\\
    0 = \Delta w_\ep +u_\ep+v_\ep-\ol{M_\ep}(t), 
    &x\in\Omega,\ t>0,\\
    \nabla u_\ep \cdot\nu = \nabla v_\ep \cdot\nu = \nabla w_\ep \cdot\nu = 0, 
    &x\in\pa\Omega,\ t>0,\\
    u_\ep(x,0)=u_{0\ep}(x), \ \quad v_\ep(x,0)=v_{0\ep}(x), 
    &x\in \Omega, 
\end{cases}
\end{align}
with $\int_\Omega w_\ep(x,t)\,dx=0$, $t>0$, where $\ep\in(0,1)$ and 
$\ol{M_\ep}(t):=\frac{1}{|\Omega|}\int_{\Omega}(u_\ep(x,t)+v_\ep(x,t))\,dx$ 
for $t>0$. 
Moreover, $u_{0\ep}\in C^\infty(\ol{\Omega})$ and $v_{0\ep}\in C^\infty(\ol{\Omega})$ are defined as
\[
  u_{0\ep} := \left.(\rho_\ep * \overline{u_0})\right|_{\overline{\Omega}}, \quad
  v_{0\ep} := \left.(\rho_\ep * \overline{v_0})\right|_{\overline{\Omega}},
\]
where $\overline{u_0}$ and $\overline{v_0}$ are the zero extensions of 
$u_0 \in L^\infty(\Omega)$ and $v_0 \in L^\infty(\Omega)$, 
that is,
  \[
    \overline{u_0}(x) := 
      \begin{cases}
        u_0(x) &\mbox{if}\ x \in \Omega,
      \\
        0 &\mbox{otherwise},
      \end{cases}
    \quad
    \overline{v_0}(x) := 
      \begin{cases}
        v_0(x) &\mbox{if}\ x \in \Omega,
      \\
        0 &\mbox{otherwise},
      \end{cases}
  \]
and $\rho_\ep \in C_{\rm c}^\infty(\mathbb{R}^n)$ 
is the mollifier such that $0 \le \rho_\ep \in C_{\rm c}^\infty(\mathbb{R}^n)$ 
is radially symmetric and that 
$\mbox{supp}\, \rho_\ep \subset \overline{B_{\ep}(0)}$ and 
$\int_{\mathbb{R}^n} \rho_\ep(x)\,dx=1$.

Unless otherwise stated we fix $M_0>0$,  
$m_1,m_2>1$, $\chi_1,\chi_2,\mu_1,\mu_2,a_1,a_2>0$ and $\ep\in(0,1)$
as well as initial data $u_0$ and $v_0$ with 
$\int_\Omega u_0\le M_0$ satisfying \eqref{initial3}.
Then, clearly $u_{0\ep}$ and $v_{0\ep}$ are also nonnegative, radially symmetric.
Hence, Lemma \ref{LS} provides the approximate solution 
$(u_\ep,v_\ep,w_\ep)$ of \eqref{P3}, 
where we denote its maximal existence time by $T_{\rm max,\ep}$. 
In order to consider convergence of approximate solution 
$(u_\ep,v_\ep,w_\ep)$ as $\ep\to0$ 
we first show uniform lower bound of $T_{\rm max,\ep}$, 
which ensures that $T_{\rm max,\ep}$ does not vanish in the passage to the limit. 

%====================================================
%                                     Lemma 3.3                                     
%====================================================
\begin{lem}\label{lem3.3}
There exist $T_0\in(0,\infty)$ and $K_0>0$ such that for any $\ep\in(0,1)$,
  \begin{align}\label{uniform}
    T_0\le T_{\rm max,\ep}
    \quad\mbox{and}\quad
    \|u_\ep(\cdot,t)\|_{L^\infty(\Omega)} 
    + \|v_\ep(\cdot,t)\|_{L^\infty(\Omega)} \le K_0
    \quad\mbox{for all } t\in[0,T_0).
  \end{align}
\end{lem}
%====================================================
\begin{proof}
The proof is based on that of \cite[Lemma 2.4]{I-Y_2013}. 
Let us take 
  \[
    p>\max\left\{1,\frac{n}{2}(2-\min\{m_1,m_2\})\right\}
  \]
and put 
  \[
    \tau_\ep := \sup\{ \tau \in (0,T_{\rm max,\ep}) \mid 
    \|u_\ep(\cdot,t)\|^p_{L^p(\Omega)} + \|v_\ep(\cdot,t)\|^p_{L^p(\Omega)} \le c_1 
    \quad\mbox{for all}\ t \in [0,\tau) \}
  \]
with 
  \begin{align}\label{c1}
    c_1:=(\|u_0\|_{L^p(\Omega)}+|\Omega|^\frac{1}{p})^p +(\|v_0\|_{L^p(\Omega)}+|\Omega|^\frac{1}{p})^p+1.
  \end{align}
Here, we have $\tau_\ep>0$ because  
$u_\ep \in C^0(\overline{\Omega}\times[0,T_{\rm max,\ep})) \subset C^0([0,T_{\rm max,\ep}); L^p(\Omega))$. 
We first confirm that there exists $T_p\in(0,\infty)$ such that 
  \begin{align}\label{Tp}
    T_p\le T_{\rm max,\ep}
    \quad\mbox{and}\quad
    \|u_\ep(\cdot,t)\|^p_{L^p(\Omega)} + \|v_\ep(\cdot,t)\|^p_{L^p(\Omega)} 
    \le c_1
    \quad\mbox{for all } t\in[0,T_p).
  \end{align}
As in the proof of \cite[Lemma 2.4]{I-Y_2013}, 
we consider the cases that $\tau_\ep=T_{\rm max,\ep}=\infty$ and that 
  \begin{align}\label{case2}
    \tau_\ep<T_{\rm max,\ep}
    \quad\mbox{and}\quad
    \|u_\ep(\cdot,\tau_\ep)\|_{L^p(\Omega)}^p 
    + \|v_\ep(\cdot,\tau_\ep)\|_{L^p(\Omega)}^p 
    = c_1.
  \end{align}
In the case that $\tau_\ep=T_{\rm max,\ep}=\infty$, 
owing to the definition of $\tau_\ep$, we obtain the estimate 
$\|u_\ep(\cdot,t)\|^p_{L^p(\Omega)} + \|v_\ep(\cdot,t)\|^p_{L^p(\Omega)} \le c_1$ 
for all $t\in[0,\infty)$. This enables us to choose any $T_p\in(0,\infty)$ 
satisfying \eqref{Tp}. 
In the case \eqref{case2}, noting that
$\mu_1 u_\ep (1-u_\ep-a_1v_\ep) \le \mu_1 u_\ep$, 
we can observe that
  \begin{align*}
    \frac{1}{p}\cdot\frac{d}{dt}\|u_\ep+\ep\|_{L^p(\Omega)}^p
    &\le -\frac{4m_1(p-1)}{(m_1+p-1)^2} 
             \| \nabla(u_\ep+\ep)^\frac{p+m_1-1}{2} \|^2_{L^2(\Omega)}
    \\ 
    &\quad\,
       +(p-1)\chi_1 \int_\Omega 
                       \nabla \left(\int^{u_\ep}_0 \xi(\xi+\ep)^{p-2}\,d\xi\right)
                       \cdot \nabla w_\ep
    \\ 
    &\quad\,
       +\mu_1 \int_\Omega u_\ep(u_\ep+\ep)^{p-1}
    \\ 
    &=: -\mathcal{I}_1 + \mathcal{I}_2 + \mathcal{I}_3
  \end{align*}
for all $t\in(0,\tau_\ep)$ (for the details see \cite[(2.5)]{TY_2022_DCDS-B}). 
As to $\mathcal{I}_2$, integrating by parts and using the third equation in \eqref{P3} and Young's inequality, we can make sure that 
  \begin{align*}
    \mathcal{I}_2 
   &= -(p-1)\chi_1 \int_\Omega  
         \left(\int^{u_\ep}_0 \xi(\xi+\ep)^{p-2}\,d\xi\right)
         \Delta w_\ep \\
   &\le (p-1)\chi_1 \int_\Omega  
         \left(\int^{u_\ep}_0 \xi(\xi+\ep)^{p-2}\,d\xi\right)
         (u_\ep + v_\ep) \\
   &\le \frac{p-1}{p}\chi_1 \int_\Omega (u_\ep+\ep)^{p+1}
          + \frac{p-1}{p}\chi_1 \int_\Omega (u_\ep+\ep)^p(v_\ep+\ep) \\
   &\le c_2 \|u_\ep+\ep\|_{L^{p+1}(\Omega)}^{p+1}
          + c_3 \|v_\ep+\ep\|_{L^{p+1}(\Omega)}^{p+1},
  \end{align*}
where $c_2:=\big(\frac{p-1}{p}+\frac{p-1}{p+1}\big) \chi_1$ 
and $c_3:=\frac{p-1}{p}\cdot\frac{1}{p+1}\chi_1$. 
Moreover, 
  \[ 
    \mathcal{I}_3\le \mu_1\|u_\ep+\ep\|_{L^p(\Omega)}^p.
  \]
Therefore, it follows that 
  \begin{align*}
    &\frac{1}{p}\cdot\frac{d}{dt}\|u_\ep+\ep\|_{L^p(\Omega)}^p \\
    &\quad\,
    \le -\mathcal{I}_1
         +c_2 \|u_\ep+\ep\|_{L^{p+1}(\Omega)}^{p+1}
          + c_3 \|v_\ep+\ep\|_{L^{p+1}(\Omega)}^{p+1}
         +\mu_1\|u_\ep+\ep\|_{L^p(\Omega)}^p
  \end{align*}
for all $t\in(0,\tau_\ep)$. 
A similar estimate for $v_\ep$ is derived, 
which together with this inequality provides us with the estimate 
such that 
  \begin{align}\label{energy}
    &\frac{1}{p}\cdot\frac{d}{dt}\|u_\ep+\ep\|_{L^p(\Omega)}^p
    + \frac{1}{p}\cdot\frac{d}{dt}\|v_\ep+\ep\|_{L^p(\Omega)}^p \notag\\
    &\quad\,
    \le -\mathcal{I}_1 - \mathcal{I}_4
         +c_4 \|u_\ep+\ep\|_{L^{p+1}(\Omega)}^{p+1}
          + c_4 \|v_\ep+\ep\|_{L^{p+1}(\Omega)}^{p+1}  \notag\\
    &\quad\,\quad\,
         +\mu_1\|u_\ep+\ep\|_{L^p(\Omega)}^p
         +\mu_2\|v_\ep+\ep\|_{L^p(\Omega)}^p
  \end{align}
for all $t\in(0,\tau_\ep)$, where 
$c_4>0$ and 
  \[
    \mathcal{I}_4:=\frac{4m_2(p-1)}{(m_2+p-1)^2} 
             \| \nabla(v_\ep+\ep)^\frac{p+m_2-1}{2} \|^2_{L^2(\Omega)}.
  \]
Applying the Gagliardo--Nirenberg inequality and 
Young's inequality to the third and fourth terms 
on the right-hand side of \eqref{energy} 
as in \cite[p.\ 269]{TY_2022_DCDS-B}, we obtain $c_5>0$ such that
  \[
    c_4 \|u_\ep+\ep\|_{L^{p+1}(\Omega)}^{p+1} 
    \le \mathcal{I}_1 
         + c_5\left\{ \| (u_\ep+\ep)^\frac{p+m_1-1}{2} \|_{L^\frac{2p}{p+m_1-1}(\Omega)}^{(1-\beta_1)\frac{2(p+1)}{p+m_1-1}} \right\}^{\theta_1}
         + c_5\| (u_\ep+\ep)^\frac{p+m_1-1}{2} \|_{L^\frac{2p}{p+m_1-1}(\Omega)}^\frac{2(p+1)}{p+m_1-1}
  \]
and
  \[
    c_4 \|v_\ep+\ep\|_{L^{p+1}(\Omega)}^{p+1}
    \le \mathcal{I}_4 
         + c_5\left\{ \| (v_\ep+\ep)^\frac{p+m_2-1}{2} \|_{L^\frac{2p}{p+m_2-1}(\Omega)}^{(1-\beta_2)\frac{2(p+1)}{p+m_2-1}} \right\}^{\theta_2}
         + c_5\| (v_\ep+\ep)^\frac{p+m_2-1}{2} \|_{L^\frac{2p}{p+m_2-1}(\Omega)}^\frac{2(p+1)}{p+m_2-1}
  \]
for all $t\in(0,\tau_\ep)$, where 
  \[
    \beta_i:=\frac{\frac{p+m_i-1}{2p}-\frac{p+m_i-1}{2(p+1)}}
               {\frac{p+m_i-1}{2p}+\frac{1}{n}-\frac{1}{2}}
    \quad\mbox{and}\quad
    \theta_i:=\left(1-\frac{\beta_i(p+1)}{p+m_i-1}\right)^{-1},
    \quad i\in\{1,2\}.
  \]
Therefore, 
  \begin{align*}
    &\frac{d}{dt}\|u_\ep+\ep\|_{L^p(\Omega)}^p
    + \frac{d}{dt}\|v_\ep+\ep\|_{L^p(\Omega)}^p \notag\\
    &\quad\,
    \le p \left(
         c_5 \|u_\ep+\ep\|_{L^p(\Omega)}^{(1-\beta_1)(p+1)\theta_1}
         + c_5 \|u_\ep+\ep\|_{L^p(\Omega)}^{p+1} 
         + \mu_1\|u_\ep+\ep\|_{L^p(\Omega)}^p
         \right) \notag\\
    &\quad\,\quad\,
         + p \left(
         c_5 \|v_\ep+\ep\|_{L^p(\Omega)}^{(1-\beta_2)(p+1)\theta_2} 
         + c_5 \|v_\ep+\ep\|_{L^p(\Omega)}^{p+1} 
         + \mu_2\|v_\ep+\ep\|_{L^p(\Omega)}^p
         \right)
  \end{align*}
for all $t\in(0,\tau_\ep)$. 
The definition of $\tau_\ep$ yields
  \[
    \| u_\ep(\cdot,t) + \ep \|_{L^p(\Omega)} 
    \le \| u_\ep(\cdot,t) \|_{L^p(\Omega)} + |\Omega|^\frac{1}{p}
    \le c_1^\frac{1}{p} + |\Omega|^\frac{1}{p}
    =:c_6
  \]
and similarly 
$\| v_\ep(\cdot,t) + \ep \|_{L^p(\Omega)} \le c_6$ for all $t\in(0,\tau_\ep)$. 
Thus we see that 
  \[
    \frac{d}{dt}\|u_\ep+\ep\|_{L^p(\Omega)}^p
    + \frac{d}{dt}\|v_\ep+\ep\|_{L^p(\Omega)}^p
    \le c_7
  \]
for all $t\in(0,\tau_\ep)$, where 
$c_7:=p\big(
         c_5c_6^{(1-\beta_1)(p+1)\theta_1} 
         + c_5c_6^{(1-\beta_2)(p+1)\theta_2}
         + 2c_5c_6^{p+1} +(\mu_1+\mu_2)c_6^p
        \big)$.
Integrating this from $0$ to $\tau_\ep$ implies that
  \[
    \|u_\ep(\cdot,\tau_\ep)\|_{L^p(\Omega)}^p
    + \|v_\ep(\cdot,\tau_\ep)\|_{L^p(\Omega)}^p
    - \|u_{0\ep}+\ep\|_{L^p(\Omega)}^p
    - \|v_{0\ep}+\ep\|_{L^p(\Omega)}^p
    \le c_7\tau_\ep.
  \]
By making use of \eqref{case2} and the inequalities 
$\| u_{0\ep}+\ep \|^p_{L^p(\Omega)} 
\le (\| u_{0} \|_{L^p(\Omega)} + |\Omega|^\frac{1}{p})^p$ 
and 
$\| v_{0\ep}+\ep \|^p_{L^p(\Omega)} 
\le (\| v_{0} \|_{L^p(\Omega)} + |\Omega|^\frac{1}{p})^p$, 
we can verify that 
  \[
    c_1
    -\left[(\| u_{0} \|_{L^p(\Omega)} + |\Omega|^\frac{1}{p})^p
      +(\| v_{0} \|_{L^p(\Omega)} + |\Omega|^\frac{1}{p})^p
     \right]
    \le c_7\tau_\ep.
  \]
Invoking \eqref{c1}, we obverse that $\frac{1}{c_7}\le \tau_\ep$, 
which means that $T_p= \frac{1}{c_7}$. 
In conclusion, we can find $T_p\in(0,\infty)$ satisfying \eqref{Tp}. 
Finally, we apply the Moser iteration (see \cite[Lemma A.1]{T-W}) to conclude this lemma. 
\end{proof}

To consider convergence of approximate solutions, 
we derive some estimates.
We first show uniform estimates for 
$\nabla (u_\ep+\ep)^{m_1}$ and $\nabla (v_\ep+\ep)^{m_2}$. 

%====================================================
%                                     Lemma 3.4                              
%====================================================
\begin{lem}\label{esti1}
Let $T_0>0$ and $K_0>0$ satisfy \eqref{uniform}. 
Then there exists $C>0$ such that for any $\ep\in(0,1)$,
  \begin{align}\label{3.4esti}
    \| \nabla (u_\ep+\ep)^{m_1} \|_{L^2(0,T_0;L^2(\Omega))} \le C
    \quad\mbox{and}\quad
    \| \nabla (v_\ep+\ep)^{m_2} \|_{L^2(0,T_0;L^2(\Omega))} \le C.
  \end{align}
\end{lem}
%====================================================
\begin{proof}
We multiply the first equation in \eqref{P3} by $u_\ep$ and
integrate it over $\Omega$ to obtain 
  \[
    \frac{1}{2} \cdot \frac{d}{dt}\| u_\ep \|_{L^2(\Omega)}^2
    \le -\frac{4m_1}{(m_1+1)^2}\| \nabla(u_\ep+\ep)^\frac{m_1+1}{2} \|_{L^2(\Omega)}^2
           +\chi_1 \int_\Omega u_\ep
             \nabla w_\ep \cdot \nabla u_\ep
           +\mu_1 \|u_\ep\|_{L^2(\Omega)}^2
  \]
for all $t\in(0,T_{\rm max,\ep})$. 
As to the second term on the right-hand side of this inequality, 
we see from integration by parts and the third equation 
in \eqref{P3} as well as \eqref{uniform} that
  \[
    \chi_1 \int_\Omega u_\ep
             \nabla w_\ep \cdot \nabla u_\ep
    \le \frac{\chi_1}{2} \int_\Omega u_\ep^3
         + \frac{\chi_1}{2} \int_\Omega u_\ep^2 v_\ep
    \le \chi_1 K_0^3|\Omega|
  \]
for all $t\in(0,T_0)$. 
The rest of this proof is the same procedure as in the proof of 
\cite[Lemma 2.4]{TY_2022_DCDS-B}, and thus we have the first 
estimate in \eqref{3.4esti}.
Similarly, the corresponding estimate for $v_\ep$ 
can be derived.
\end{proof}

We next establish the following estimates for $u_\ep$ and $v_\ep$.

%====================================================
%                                     Lemma 3.5                              
%====================================================
\begin{lem}\label{esti2}
Let $T_0>0$ and $K_0>0$ satisfy \eqref{uniform}. 
Then there exists $C>0$ such that for any $\ep\in(0,1)$,
  \begin{align}\label{inequep}
    \| \sqrt{t} (u_\ep^{m_1})_t \|_{L^2(0,T_0;L^2(\Omega))}^2 
    + \sup_{t\in(0,T_0)} \| \sqrt{t} \nabla u_\ep^{m_1}(\cdot,t) \|_{L^2(\Omega)}^2 \le C
  \end{align}
and 
  \begin{align}\label{ineqvep}
    \| \sqrt{t} (v_\ep^{m_2})_t \|_{L^2(0,T_0;L^2(\Omega))}^2 
    + \sup_{t\in(0,T_0)} \| \sqrt{t} \nabla v_\ep^{m_2}(\cdot,t) \|_{L^2(\Omega)}^2 \le C.
  \end{align}
\end{lem}
%====================================================
\begin{proof}
In deriving the claim, we follow \cite[Lemma 2.5]{TY_2022_DCDS-B}. 
In the following, we mainly point out the differences.
By multiplying the first equation in \eqref{P3} 
by $[(u_\ep+\ep)^m]_t$ 
and integrating it over $\Omega$, 
we can observe that 
  \begin{align*}
    &\frac{4m_1}{(m_1+1)^2}
      \left\| \left[(u_\ep+\ep)^\frac{m_1+1}{2}\right]_t \right\|_{L^2(\Omega)}^2
    +\frac{1}{2} \cdot \frac{d}{dt} \| \nabla (u_\ep+\ep)^{m_1} \|_{L^2(\Omega)}^2
    \\
    &\quad\,
      \le -\frac{2m_1}{m_1+1}\int_\Omega
             \big[
                    \nabla u_\ep \cdot\nabla w_\ep
                    +u_\ep \Delta w_\ep
             \big]
             (u_\ep+\ep)^\frac{m_1-1}{2}
             \left[(u_\ep+\ep)^\frac{m_1+1}{2}\right]_t
    \\
    &\quad\,\quad\,
           +\mu_1 \int_\Omega u_\ep \left[(u_\ep+\ep)^{m_1}\right]_t
           -\mu_1 \int_\Omega u_\ep^2 \left[(u_\ep+\ep)^{m_1}\right]_t
           -\mu_1a_1 \int_\Omega u_\ep v_\ep \left[(u_\ep+\ep)^{m_1}\right]_t
  \end{align*}
for all $t\in(0,T_{\rm max,\ep})$. 
The third equation in \eqref{P3} entails that 
$v_\ep=-\Delta w_\ep - u_\ep +\ol{M}(t)$. 
Thus the last term on the right-hand side of the above inequality 
is rewritten as 
  \begin{align*}
    -\mu_1a_1 \int_\Omega u_\ep v_\ep \left[(u_\ep+\ep)^{m_1}\right]_t
    &= \mu_1a_1 \int_\Omega u_\ep \Delta w_\ep \left[(u_\ep+\ep)^{m_1}\right]_t
      + \mu_1a_1 \int_\Omega u_\ep^2 \left[(u_\ep+\ep)^{m_1}\right]_t \\
    &\quad\,
      - \mu_1a_1 \ol{M}(t) \int_\Omega u_\ep \left[(u_\ep+\ep)^{m_1}\right]_t, 
  \end{align*}
so that we have 
  \begin{align*}
    &\frac{4m_1}{(m_1+1)^2}
      \left\| \left[(u_\ep+\ep)^\frac{m_1+1}{2}\right]_t \right\|_{L^2(\Omega)}^2
      + \frac{1}{2} \cdot \frac{d}{dt} \| \nabla (u_\ep+\ep)^{m_1} \|_{L^2(\Omega)}^2
    \\
    &\quad\, 
    \le -\frac{2m_1}{m_1+1}\int_\Omega
                    \nabla u_\ep \cdot\nabla w_\ep
             (u_\ep+\ep)^\frac{m_1-1}{2}
             \left[(u_\ep+\ep)^\frac{m_1+1}{2}\right]_t
    \\
    &\quad\,\quad\, 
           -\frac{2m_1}{m_1+1}\left(1 - \mu_1a_1\right) \int_\Omega
                    u_\ep \Delta w_\ep
             (u_\ep+\ep)^\frac{m_1-1}{2}
             \left[(u_\ep+\ep)^\frac{m_1+1}{2}\right]_t
    \\
    &\quad\,\quad\, 
           +\mu_1 \int_\Omega u_\ep \left[(u_\ep+\ep)^{m_1}\right]_t
           -\mu_1 \int_\Omega u_\ep^2 \left[(u_\ep+\ep)^{m_1}\right]_t
    \\
    &\quad\,\quad\, 
           + \mu_1a_1 \int_\Omega u_\ep^2 \left[(u_\ep+\ep)^{m_1}\right]_t
           - \mu_1a_1 \ol{M}(t) \int_\Omega u_\ep \left[(u_\ep+\ep)^{m_1}\right]_t
  \end{align*}
for all $t\in(0,T_{\rm max,\ep})$. 
Thanks to \eqref{uniform}, we see from elliptic regularity theory applied 
to the third equation in \eqref{P3} and the Sobolev embedding theorem 
that there is $c_1>0$ such that
$\|\nabla w_\ep(\cdot,t)\|_{L^\infty(\Omega)}\le c_1$
for all $t\in(0,T_0)$, where 
$c_1$ is independent of $\ep\in(0,1)$. 
Furthermore, the third equation in \eqref{P3} and \eqref{uniform} 
imply that 
  \[
     |\Delta w_\ep| \le \big| 
\frac{1}{|\Omega|}\int_\Omega (u_\ep+v_\ep)  
- u_\ep -v_\ep \big| \le 2K_0
  \]
for all $t\in(0,T_0)$. 
Hence, we argue as in \cite[pp.\ 272, 273]{TY_2022_DCDS-B} to obtain 
$c_2>0$ and $c_3>0$ such that 
  \begin{align*}
    &\frac{2m_1}{(m_1+1)^2}
      \left\| \left[(u_\ep+\ep)^\frac{m_1+1}{2}\right]_t \right\|_{L^2(\Omega)}^2
      + \frac{1}{2} \cdot \frac{d}{dt} \| \nabla (u_\ep+\ep)^{m_1} \|_{L^2(\Omega)}^2
    \\
    &\quad\, 
    \le c_2 \left\| \nabla (u_\ep+\ep)^\frac{m_1+1}{2} \right\|_{L^2(\Omega)}^2
         + c_3
    \\
    &\quad\,\quad\, 
           + m_1 \mu_1 \frac{d}{dt}\int_\Omega
                \left(
                  \int^{u_\ep}_0  \xi(\xi+\ep)^{m_1-1}\,d\xi
                \right)
           - m_1 \mu_1 \frac{d}{dt}\int_\Omega
                \left(
                  \int^{u_\ep}_0  \xi^2(\xi+\ep)^{m_1-1}\,d\xi
                \right)
    \\
    &\quad\,\quad\, 
           + m_1 \mu_1 \frac{d}{dt}\int_\Omega
                \left(
                  \int^{u_\ep}_0  \xi^2(\xi+\ep)^{m_1-1}\,d\xi
                \right)
           - \mu_1a_1 \ol{M}(t) \frac{d}{dt}\int_\Omega
                \left(
                  \int^{u_\ep}_0  \xi(\xi+\ep)^{m_1-1}\,d\xi
                \right)
  \end{align*}
for all $t\in(0,T_0)$. 
Multiplying this by $t$, changing the variable $t$ with $s$ 
and integrating it over $(0,t)$, we can make sure that 
  \begin{align*}
    &\frac{2m_1}{(m_1+1)^2}
      \left\| \sqrt{s} \left[(u_\ep+\ep)^\frac{m_1+1}{2}\right]_t \right\|_{L^2(0,t;L^2(\Omega))}^2
      + \frac{1}{2} t \| \nabla (u_\ep+\ep)^{m_1} \|_{L^2(\Omega)}^2
    \\
    &\quad\,
    \le \frac{1}{2} \| \nabla (u_\ep+\ep)^{m_1} \|_{L^2(0,t;L^2(\Omega))}^2
         + c_2 \left\| \sqrt{s}\nabla (u_\ep+\ep)^\frac{m_1+1}{2} \right\|_{L^2(0,t;L^2(\Omega))}^2 
         + c_3t 
    \\
    &\quad\,\quad\,
         + m_1 \mu_1 t \int_\Omega
                \left(
                  \int^{u_\ep}_0  \xi(\xi+\ep)^{m_1-1}\,d\xi
                \right)
         + m_1 \mu_1 \int^t_0 \left[\int_\Omega
                \left(
                  \int^{u_\ep}_0  \xi^2(\xi+\ep)^{m_1-1}\,d\xi
                \right)  \right]
    \\
    &\quad\,\quad\,
         + m_1 \mu_1 t \int_\Omega
                \left(
                  \int^{u_\ep}_0  \xi^2(\xi+\ep)^{m_1-1}\,d\xi
                \right)
    \\
    &\quad\,\quad\,
         + m_1 \mu_1a_1 \int^t_0 
                \left[ (\ol{M}(t)+t\ol{M}'(t)) \int_\Omega
                \left(
                  \int^{u_\ep}_0  \xi^2(\xi+\ep)^{m_1-1}\,d\xi
                \right)  \right]
%    \\
%    &\quad\,
%    =: \widetilde{I}_1+\widetilde{I}_2+\widetilde{I}_3+\widetilde{I}_4+\widetilde{I}_5+\widetilde{I}_6+\widetilde{I}_7
  \end{align*}
for all $t\in(0,T_0)$. 
With regard to the sixth term, \eqref{uniform} yields 
  \begin{align*}
    m_1 \mu_1 t \int_\Omega
                \left(
                  \int^{u_\ep}_0  \xi^2(\xi+\ep)^{m_1-1}\,d\xi
                \right)
    &\le \frac{m_1 \mu_1}{m_1+2} t \int_\Omega (u_\ep+\ep)^{m_1+2} \\
    &\le \frac{m_1 \mu_1}{m_1+2} (K_0+1)^{m_1+2}|\Omega|T_0
  \end{align*}
for all $t\in(0,T_0)$. 
As to the seventh term, since \eqref{uniform} and the first and second equations in \eqref{P3} imply that 
$\ol{M}(t)\le K_0$ and 
$\ol{M}'(t)\le \mu_1\int_\Omega(u+v) \le \mu_1K_0$, 
it follows that 
  \begin{align*}
         &m_1 \mu_1a_1 \int^t_0 
                \left[ (\ol{M}(t)+t\ol{M}'(t)) \int_\Omega
                \left(
                  \int^{u_\ep}_0  \xi^2(\xi+\ep)^{m_1-1}\,d\xi
                \right)  \right] \\
    \quad\,&
      \le \frac{m_1 \mu_1a_1}{m_1+2} (1+\mu_1T_0)K_0 \int^t_0 
         \left( \int_\Omega (u_\ep+\ep)^{m_1+2} \right)\\
    \quad\,&
      \le \frac{m_1 \mu_1a_1}{m_1+2} (1+\mu_1T_0)K_0  
         (K_0+1)^{m_1+2}|\Omega|T_0
  \end{align*}
for all $t\in(0,T_0)$. 
The other terms are bounded by some constant, 
and so we can attain \eqref{inequep} 
(for details see \cite[pp.\ 273, 274]{TY_2022_DCDS-B}). 
Similarly, \eqref{ineqvep} holds.
\end{proof}

We are now in a position to prove convergence of approximate solutions. 

%====================================================
%                                     Lemma 3.6                              
%====================================================
\begin{lem}\label{lem3.6}
Let $T_0>0$ and $K_0>0$ satisfy \eqref{uniform}.  
Then there exist subsequences $\{u_{\ep_k}\}$, $\{v_{\ep_k}\}$, $\{w_{\ep_k}\}$ ($\ep_k \to 0$ as $k \to \infty$) 
and nonnegative functions $u, v, w$ such that 
  \[
    u, v \in L^\infty(0,T_0;L^\infty(\Omega)), \ 
    u^{m_1}, v^{m_2} \in L^2(0,T_0;H^1(\Omega)),
    w \in L^\infty(0,T_0;W^{1,\infty}(\Omega)),
  \]
and as $k\to0$, 
  \begin{align*}
    &u_{\ep_k} \to u, \ v_{\ep_k} \to v 
    \quad\mbox{weakly}^\star\ \mbox{in}\ L^\infty(0,T_0;L^\infty(\Omega)),
    \\[1mm]
    &u_{\ep_k} \to u, \ v_{\ep_k} \to v \quad\mbox{strongly in}\ C^0([\eta,T_0];L^p(\Omega))
    \quad\mbox{for all}\ \eta\in(0,T_0) \ \mbox{and}\ p\in[1,\infty),
    \\[1mm]
    &\nabla (u_{\ep_k}+\ep_k)^{m_1} \to \nabla u^{m_1}, 
    \ \nabla (v_{\ep_k}+\ep_k)^{m_2} \to \nabla v^{m_2}
    \quad\mbox{weakly in}\ L^2(0,T_0;L^2(\Omega)),
    \\[1mm]
    &\nabla w_{\ep_k} \to \nabla w
    \quad\mbox{weakly}^\star\ \mbox{in}\ L^\infty(0,T_0;L^\infty(\Omega)).
\end{align*}
\end{lem}
%====================================================
\begin{proof}
This lemma can be shown by using Lemmas \ref{esti1} and \ref{esti2} 
as in \cite[Lemma 5.3]{I-Y_2012}. 
\end{proof}

Now we are in the position to prove Proposition \ref{LMS}.

%====================================================
%                            Proof of Proposition 3.2                              
%====================================================
\begin{proof}[Proof of Proposition \ref{LMS}] 
Let $M_0>0$ and $\widetilde{M}_0\in(0,M_0)$.
We define $U_\ep$ and $\phi_{U_\ep}$ as 
  \[
    U_\ep(s,t) := \int^{s^\frac{1}{n}}_0 \rho^{n-1} u_\ep(\rho,t)\,d\rho
    \quad\mbox{and}\quad
    \phi_{U_\ep}(t) := \int^{s_0}_{0} s^{-b}(s_0-s) U_\ep(s,t)\,ds
  \]
for $s\in[0,R^n]$ and $t\in[0,T_{\rm max})$, 
where $s_0\in(0,R^n)$, $b\in(1,2)$ and $\ep\in(0,1)$.
We see that for initial data $u_0$ and $v_0$ fulfilling \eqref{initial3}, 
$u_{0\ep}$ and $v_{0\ep}$ are also nonnegative, radially symmetric and radially nonincreasing.
Thus Lemma \ref{lem2.6} ensures that there are $b\in(1,2)$, 
$c_1>0$ and $c_2>0$ such that for any $s_0\in(0,R^n)$ and $\ep\in(0,1)$,
  \[
    \phi_{U_\ep}'(t) \ge c_1s_0^{-(3-b)} \phi_{U_\ep}^2(t) -c_2 s_0^{3-b-\frac{4}{n(2-m_1)}}
  \]
for all $t\in(0,T_{\rm max,\ep})$. 
Here, we note from Remark \ref{rem} that $c_1$ and $c_2$ are independent of $\ep$. 
Now we take $\eta_0\in(0,\widetilde{M}_0)$ so small 
and fix $s_0\in(0,R^n)$ satisfying 
  \begin{align}\label{approcondi1}
     s_0 \le \left(\frac{c_1}{2} \cdot \left(\frac{2^{2-b} (\widetilde{M}_0-\eta_0)}{\omega_n}\right)^2 \cdot \frac{1}{c_2}\right)^{\frac{1}{2-\frac{4}{n(2-m_1)}}}.
  \end{align}
Also, put $r_0:=\big(\frac{s_0}{4}\big)^\frac{1}{n}$ 
and choose $u_0$ and $v_0$ with \eqref{initial2} and \eqref{initial3}. 
The definition of $u_{0\ep}$ implies $u_{0\ep}\to u_0$ 
in $L^1(\Omega)$ as $\ep\to0$, so that 
there is $\ep_0\in(0,1)$ such that for any $\ep\in(0,\ep_0)$,
  \begin{align}\label{approcondi2}
    \int_{B_{r_0}(\Omega)} u_{0\ep} \ge \widetilde{M}_0-\eta_0.
  \end{align}
Moreover, Lemma \ref{lem3.3} provides $T_0>0$ and $K_0>0$ with \eqref{uniform}.
By using \eqref{approcondi1}, \eqref{approcondi2} and
  \[
    S_\ep:=\left\{T\in(0,T_0) \;\middle|\; \phi_{U_\ep}(t)\ge\frac{2^{b-3}(\widetilde{M}_0-\eta_0)}{\omega_n}s_0^{2-b}\quad\mbox{for all}\ t\in[0,T]\right\}, 
  \]
as in the proof of Lemma \ref{lem2.7} we have $c_3>0$ such that 
$\phi_{U_\ep}'(t) \ge c_3s_0^{-(3-b)} \phi_{U_\ep}^2(t)$ for all $t\in(0,T_0)$ and $\ep\in(0,\ep_0)$. 
Integrating it over $(0,t)$ implies 
$\phi_{U_\ep}(t) - \phi_{U_\ep}(0) \ge c_3 \int^t_0 \phi_{U_\ep}^2(\tau)\,d\tau$
for all $t\in(0,T_0)$ and $\ep\in(0,\ep_0)$. 
By Lemma \ref{lem3.6}, the rest of the proof 
is a similar procedure as in \cite[Proof of Proposition 1]{TY_2022_DCDS-B}, and then, 
the conclusion results from Definition \ref{mms}. 
\end{proof}

%====================================================
%                                                                                                             
%                                  Subsection 3.2                                  
%                                                                                                             
%====================================================
\subsection{Finite-time blow-up}

Before proving Theorem \ref{thm2}, we state the following equivalence (see \cite[Lemma 3.1]{TY_2022_DCDS-B}).

%====================================================
%                                     Lemma 3.7                              
%====================================================
\begin{lem}\label{lem3.7}
Let $T\in(0,\infty)$. Assume that a triplet $(u,v,w)$ of nonnegative 
functions defined on $\Omega \times (0,T)$ satisfies 
that $u,v \in L^\infty(0,T;L^\infty(\Omega))\cap C^0_{{\rm w}-\star}([0,T);L^\infty(\Omega))$ and 
$u^{m_1},v^{m_2}, w \in L^2(0,T;H^1(\Omega))$.
Then the following two conditions are equivalent.
\begin{enumerate}
%%%%%%%%%%%%%%%%%%%%%%%
\setlength{\itemindent}{-7pt}
%%%%%%%%%%%%%%%%%%%%%%%
\item[{\rm (a)}] For all $\varphi \in L^2(0,T;H^1(\Omega)) \cap W^{1,1}(0,T;L^2(\Omega))$
with ${\rm supp}\, \varphi(x,\cdot)\subset[0,T)$ {\rm (}a.a.\ $x\in\Omega${\rm )},
\begin{align*}
  &\int^{T}_{0} \int_\Omega 
    (\nabla u^{m_1} \cdot \nabla \varphi - \chi_1 u \nabla w \cdot \nabla \varphi
     - \mu_1 u (1-u-a_1v) \varphi 
     - u \varphi_t)
  = \int_\Omega u_0 \varphi(\cdot,0),
  \\
  &\int^{T}_{0} \int_\Omega 
    (\nabla v^{m_2} \cdot \nabla \varphi - \chi_2 v \nabla w \cdot \nabla \varphi
     - \mu_2 v (1-a_2u-v) \varphi 
     - v \varphi_t)
  = \int_\Omega v_0 \varphi(\cdot,0),
  \\
  &\int^{T}_{0} \int_\Omega \nabla w \cdot \nabla \varphi
  - \int^{T}_{0} \int_\Omega (u+v) \varphi 
  + \int^{T}_{0} \left(\ol{M}(t) \int_\Omega \varphi \right)
  =0,
\end{align*}
  \item[{\rm (b)}] $u_t, v_t \in L^2(0,T;(H^1(\Omega))^\star)$,
  and for all $\psi \in H^1(\Omega)$,
    \begin{align*}
      &\int_\Omega u_t \psi
      = -\int_\Omega (\nabla u^{m_1} \cdot \nabla \psi 
                              - \chi_1 u \nabla w \cdot \nabla \psi
                              - \mu_1u(1-u-a_1v) \psi),
    \\
      &\int_\Omega v_t \psi
      = -\int_\Omega (\nabla v^{m_2} \cdot \nabla \psi 
                              - \chi_2 v \nabla w \cdot \nabla \psi
                              - \mu_2v(1-a_2u-v) \psi),
    \\
      &\int_\Omega \nabla w \cdot \nabla \psi
      - \int_\Omega (u+v) \psi 
      + \overline{M}(t) \int_\Omega \psi= 0
    \end{align*}
\hspace{-7pt}for a.a.\ $t \in [0,T)$ with $u(\cdot,0)=u_0$ and $v(\cdot,0)=v_0$.
  \end{enumerate}
\end{lem}

Finally, we prove Theorem \ref{thm2}.

%====================================================
%                            Proof of Theorem 3.1                              
%====================================================
\begin{proof}[Proof of Theorem \ref{thm2}] 
This theorem can be shown similarly to \cite[Theorem 1.4]{TY_2022_DCDS-B}. 
Thus we only give a sketch of the proof. 
Let $M_0>0$ and $\widetilde{M}_0\in(0,M_0)$. 
Moreover, let $r_0\in(0,R)$ provided by Proposition \ref{LMS} 
and take $u_0$ and $v_0$ satisfying \eqref{initial2} and \eqref{initial3}.
Then, thanks to Proposition \ref{LMS}, we have 
a maximal moment solution $(u,v,w)$ of \eqref{P2} on $[0,\widetilde{T}_{\rm max})$. 
Then there exists $K>0$ such that 
  \begin{align}\label{key}
    \phi_U(t)-\phi_U(0) \ge K \int^t_0 \phi_U^2(\tau)\,d\tau
  \end{align}
for all $t\in(0,\widetilde{T}_{\rm max})$. 

We first verify that $\widetilde{T}_{\rm max}<\infty$ by contradiction. 
To this end, we assume that $\widetilde{T}_{\rm max}=\infty$. 
Let us introduce $\Phi_U$ as
  \[
    \Phi_U(t) := \int_{0}^{t} \phi_U^2(\tau)\,d\tau + \frac{\phi_U(0)}{K}
    \quad\mbox{for}\ t \in [0,\infty), 
  \]
where $\Phi_U \in C^0([0,\infty)) \cap C^1((0,\infty))$. 
In light of \eqref{key} we see that 
$\Phi_U'(t) \ge K^2\Phi_U^2(t)$
for all $t \in (0,\infty)$, which leads to 
$t \le \frac{1}{K^2\Phi_U(0)}<\infty$ for all $t\in(0,\infty)$. 
This is a contradiction, so that $\widetilde{T}_{\rm max}<\infty$. 

Next we prove that $(u,v,w)$ blows up at $\widetilde{T}_{\rm max}$ by contradiction. 
To this end, let us assume that
${\limsup}_{t \nearrow \widetilde{T}_{\rm max}}\, (\|u(t)\|_{L^\infty(\Omega)} + \|v(t)\|_{L^\infty(\Omega)})<\infty$.
By means of Lemma \ref{lem3.7} it follows that 
$u_t, v_t \in L^2(0,\widetilde{T}_{\rm max};(H^1(\Omega))^\star)$, 
which ensures that $u$ and $v$ are uniformly continuous 
on $[0,\widetilde{T}_{\rm max})$ in $(H^1(\Omega))^\star$. 
Hence we can find $\widetilde{u}_{\widetilde{T}_{\rm max}}, \widetilde{v}_{\widetilde{T}_{\rm max}}\in (H^1(\Omega))^\star$ such that 
  \begin{alignat*}{3}
    \widetilde{u}_{\widetilde{T}_{\rm max}} &= 
    \lim_{t \nearrow \widetilde{T}_{\rm max}} u(\cdot,t), 
    &\quad\mbox{in}\ (H^1(\Omega))^\star \cap L^\infty(\Omega), \\ 
    \widetilde{v}_{\widetilde{T}_{\rm max}} &= 
    \lim_{t \nearrow \widetilde{T}_{\rm max}} v(\cdot,t) 
    &\quad\mbox{in}\ (H^1(\Omega))^\star \cap L^\infty(\Omega).    
  \end{alignat*}
Choosing the initial data as
$\widetilde{u}_{\widetilde{T}_{\rm max}}$ and $\widetilde{v}_{\widetilde{T}_{\rm max}}$, and 
constructing a weak solution of \eqref{P2} on $[\widetilde{T}_{\rm max},\widetilde{T}_{\rm max}+T_1)$ 
with some $T_1>0$, 
we can extend $(u,v,w)$ to a weak solution on $[0,\widetilde{T}_{\rm max}+T_1)$. 
Also, by continuity of $\phi_U$ at $\widetilde{T}_{\rm max}$ 
we can verify that the moment inequality holds beyond $\widetilde{T}_{\rm max}$. 
Thus we obtain a moment solution of \eqref{P2} on $[0,\widetilde{T}_{\rm max}+\sigma_1)$ for some $\sigma_1\in(0,T_1)$, which 
contradicts the definition of $\widetilde{T}_{\rm max}$. 
Consequently, $(u,v,w)$ blows up at $\widetilde{T}_{\rm max}$.
\end{proof}

%\section*{Declarations}
%\noindent{\bf Competing Interests} The author declares no competing interests.

\newpage

\bibliographystyle{plain}
\bibliography{T4_2023_two-sp_dege_arXiv.bbl}

\begin{thebibliography}{10}

\bibitem{B-W}
X.~Bai and M.~Winkler.
\newblock Equilibration in a fully parabolic two-species chemotaxis system with
  competitive kinetics.
\newblock {\em Indiana Univ.\ Math.\ J.}, 65:553--583, 2016.

\bibitem{B-F-L}
T.~Black, M.~Fuest, and J.~Lankeit.
\newblock Relaxed parameter conditions for chemotactic collapse in
  logistic-type parabolic--elliptic {K}eller--{S}egel systems.
\newblock {\em Z. Angew.\ Math.\ Phys.}, 72(9):Paper No.\ 96, 23 pp., 2021.

\bibitem{BFLM_2022}
T.~Black, M.~Fuest, J.~Lankeit, and M.~Mizukami.
\newblock Possible points of blow-up in chemotaxis systems with spatially
  heterogeneous logistic source.
\newblock {\em Nonlinear Anal. Real World Appl.}, 73:Paper No. 103868,
  1468--1218, 2023.

\bibitem{BLM}
T.~Black, J.~Lankeit, and M.~Mizukami.
\newblock On the weakly competitive case in a two-species chemotaxis model.
\newblock {\em IMA J. Appl.\ Math.}, 81:860--876, 2016.

\bibitem{Brown}
P.~N. Brown.
\newblock Decay to uniform states in ecological interactions.
\newblock {\em SIAM J. Appl.\ Math.}, 38:22--37, 1980.

\bibitem{C-W}
T.~Cie\'{s}lak and M.~Winkler.
\newblock Finite-time blow-up in a quasilinear system of chemotaxis.
\newblock {\em Nonlinearity}, 21(5):1057--1076, 2008.

\bibitem{CS_1977}
E.~D. Conway and J.~A. Smoller.
\newblock A comparison technique for systems of reaction-diffusion equations.
\newblock {\em Comm. Partial Differential Equations}, 2:679--697, 1977.

\bibitem{MR_1979}
P.~de~Mottoni and F.~Rothe.
\newblock Convergence to homogeneous equilibrium state for generalized
  {V}olterra--{L}otka systems with diffusion.
\newblock {\em SIAM J. Appl.\ Math.}, 37:648--663, 1979.

\bibitem{F-2020}
M.~Fuest.
\newblock Finite-time blow-up in a two-dimensional {K}eller--{S}egel system
  with an environmental dependent logistic source.
\newblock {\em Nonlinear Anal.\ Real World Appl.}, 52:103022, 14 pp., 2020.

\bibitem{F_2021_optimal}
M.~Fuest.
\newblock Approaching optimality in blow-up results for {K}eller--{S}egel
  systems with logistic-type dampening.
\newblock {\em NoDEA Nonlinear Differential Equations Appl.}, 28(2):Paper No.\
  16, 17 pp., 2021.

\bibitem{HHWY_2020}
H.~Huang, R.~Huang, L.~Wang, and J.~Yin.
\newblock Periodic solutions for the degenerate {L}otka--{V}olterra competition
  system.
\newblock {\em Qual. Theory Dyn.\ Syst.}, 19(2):Paper No. 73, 35, 2020.

\bibitem{I-Y_2012}
S.~Ishida and T.~Yokota.
\newblock Global existence of weak solutions to quasilinear degenerate
  {K}eller--{S}egel systems of parabolic--parabolic type.
\newblock {\em J. Differential Equations}, 252(2):1421--1440, 2012.

\bibitem{I-Y_2013}
S.~Ishida and T.~Yokota.
\newblock Blow-up in finite or infinite time for quasilinear degenerate
  {K}eller--{S}egel systems of parabolic--parabolic type.
\newblock {\em Discrete Contin.\ Dyn.\ Syst.\ Ser.\ B}, 18(10):2569--2596,
  2013.

\bibitem{KY_1993}
Y.~Kan-on and E.~Yanagida.
\newblock Existence of non-constant stable equilibria in competition diffusion
  equations.
\newblock {\em Hiroshima Math.\ J.}, 23:193--221, 1993.

\bibitem{KW_1985}
K.~Kishimoto and H.~F. Weinberger.
\newblock The spatial homogeneity of stable equilibria of some
  reaction-diffusion systems on convex domains.
\newblock {\em J. Differential Equations}, 58:15--21, 1985.

\bibitem{LMW_2015}
K.~Lin, C.~Mu, and L.~Wang.
\newblock Boundedness in a two-species chemotaxis system.
\newblock {\em Math.\ Methods Appl.\ Sci.}, 38:5085--5096, 2015.

\bibitem{Lou-Ni}
Y.~Lou and W.-M. Ni.
\newblock Diffusion, self-diffusion and cross-diffusion.
\newblock {\em J. Differential Equations}, 131:79--131, 1996.

\bibitem{Matano-Mimura}
H.~Matano and M.~Mimura.
\newblock Pattern formation in competition-diffusion systems in nonconvex
  domains.
\newblock {\em Publ.\ Res.\ Inst.\ Math.\ Sci.}, 19:1049--1079, 1983.

\bibitem{M_2017_DCDSB}
M.~Mizukami.
\newblock Boundedness and asymptotic stability in a two-species
  chemotaxis-competition model with signal-dependent sensitivity.
\newblock {\em Discrete Contin.\ Dyn.\ Syst.\ Ser.\ B}, 22:2301--2319, 2017.

\bibitem{M_2018_MMAS}
M.~Mizukami.
\newblock Boundedness and stabilization in a two-species chemotaxis-competition
  system of parabolic--parabolic--elliptic type.
\newblock {\em Math.\ Methods Appl.\ Sci.}, 41:234--249, 2018.

\bibitem{MT_2023}
M.~Mizukami and Y.~Tanaka.
\newblock Finite-time blow-up in a two-species chemotaxis-competition model
  with single production.
\newblock {\em Arch. Math. (Brno)}, 59(2):215--222, 2023.

\bibitem{MTY_2022}
M.~Mizukami, Y.~Tanaka, and T.~Yokota.
\newblock Can chemotactic effects lead to blow-up or not in two-species
  chemotaxis-competition models?
\newblock {\em Z. Angew.\ Math.\ Phys.}, 73(6):Paper No. 239, 2022.

\bibitem{MurrayI}
J.~D. Murray.
\newblock {\em Mathematical biology. {I} An introduction}, volume~17 of {\em
  Interdisciplinary Applied Mathematics}.
\newblock Springer-Verlag, New York, third edition, 2002.

\bibitem{MurrayII}
J.~D. Murray.
\newblock {\em Mathematical biology. {II} Spatial models and biomedical
  applications}, volume~18 of {\em Interdisciplinary Applied Mathematics}.
\newblock Springer-Verlag, New York, third edition, 2003.

\bibitem{Okubo-Leivin}
A.~Okubo and S.~A. Levin.
\newblock {\em Diffusion and ecological problems: modern perspectives},
  volume~14 of {\em Interdisciplinary Applied Mathematics}.
\newblock Springer-Verlag, New York, second edition, 2001.

\bibitem{Pao_2015}
C.~V. Pao.
\newblock Dynamics of {L}otka--{V}olterra competition reaction--diffusion
  systems with degenerate diffusion.
\newblock {\em J. Math.\ Anal.\ Appl.}, 421(2):1721--1742, 2015.

\bibitem{STW}
C.~Stinner, J.~I. Tello, and M.~Winkler.
\newblock Competitive exclusion in a two-species chemotaxis model.
\newblock {\em J. Math.\ Biol.}, 68:1607--1626, 2014.

\bibitem{T_2022_JMAA}
Y.~Tanaka.
\newblock Boundedness and finite-time blow-up in a quasilinear
  parabolic--elliptic chemotaxis system with logistic source and nonlinear
  production.
\newblock {\em J. Math.\ Anal.\ Appl.}, 506:Paper No.\ 125654, 29 pp., 2022.

\bibitem{T_2023_Eq}
Y.~Tanaka.
\newblock Existence of blow-up solutions for a degenerate parabolic--elliptic
  {K}eller--{S}egel system with logistic source.
\newblock {\em Arch. Math. (Brno)}, 59(2):223--230, 2023.

\bibitem{TY_2022_DCDS-B}
Y.~Tanaka and T.~Yokota.
\newblock Finite-time blow-up in a quasilinear degenerate parabolic--elliptic
  chemotaxis system with logistic source and nonlinear production.
\newblock {\em Discrete Contin.\ Dyn.\ Syst.\ Ser.\ B}, 28(1):262--286, 2023.

\bibitem{T-W}
Y.~Tao and M.~Winkler.
\newblock Boundedness in a quasilinear parabolic--parabolic {K}eller--{S}egel
  system with subcritical sensitivity.
\newblock {\em J. Differential Equations}, 252(1):692--715, 2012.

\bibitem{T-W_2012}
J.~I. Tello and M.~Winkler.
\newblock Stabilization in a two-species chemotaxis system with a logistic
  source.
\newblock {\em Nonlinearity}, 25:1413--1425, 2012.

\bibitem{W-2011}
M.~Winkler.
\newblock Blow-up in a higher-dimensional chemotaxis system despite logistic
  growth restriction.
\newblock {\em J. Math.\ Anal.\ Appl.}, 384(2):261--272, 2011.

\bibitem{W-2018_nonlinear}
M.~Winkler.
\newblock A critical blow-up exponent in a chemotaxis system with nonlinear
  signal production.
\newblock {\em Nonlinearity}, 31(5):2031--2056, 2018.

\bibitem{W-2018}
M.~Winkler.
\newblock Finite-time blow-up in low-dimensional {K}eller--{S}egel systems with
  logistic-type superlinear degradation.
\newblock {\em Z. Angew.\ Math.\ Phys.}, 69(2):Paper No.\ 69, 40 pp., 2018.

\end{thebibliography}
\end{document}